\documentclass{amsart}[10pt]
\usepackage[utf8]{inputenc}
\usepackage[english]{babel}
\usepackage[T1]{fontenc}
\usepackage{amsmath}
\usepackage{amsfonts}
\usepackage{amssymb}
\usepackage{amsthm}
\usepackage{amscd}
\usepackage{geometry}
\usepackage{enumitem} 
\usepackage{cancel} 
\usepackage{graphicx}
\usepackage{mathtools}
\usepackage{color}
\usepackage{empheq}
\usepackage{hyperref}
\usepackage{comment}

\newcommand\R{\mathbb{R}}

\newcommand\eps{\varepsilon}

\newcommand\kin{\mathrm{kin}}

\newcommand\loc{\mathrm{loc}}

\theoremstyle{plain}
\newtheorem{theorem}{Theorem} [section]
\newtheorem{lemma}[theorem]{Lemma}

\newtheorem{proposition}[theorem]{Proposition}

\theoremstyle{remark}
\newtheorem{remark}[theorem]{Remark}

\theoremstyle{definition}
\newtheorem{definition}[theorem]{Definition}

\newtheorem{assumption}[theorem]{Assumption}

\numberwithin{equation}{section}

\title{The isothermal limit for the compressible Euler Equations with Damping}

\author{Quentin Chauleur}
\address{Univ Rennes, CNRS, IRMAR - UMR 6625, F-35000 Rennes, France}
\email{quentin.chauleur@ens-rennes.fr}

\begin{document}

\maketitle

\begin{abstract}
We consider the isothermal Euler system with damping. We rigorously show the convergence of Barenblatt solutions towards a limit Gaussian profile in the isothermal limit $\gamma \rightarrow 1$, and we explicitly compute the propagation and the behavior of Gaussian initial data. We then show the weak $L^1$ convergence of the density as well as the asymptotic behavior of its first and second moments.
\end{abstract}

\tableofcontents

\section{Introduction}

We consider the isentropic compressible Euler equations with frictional damping:
\begin{subequations} \label{EL_baro}
\begin{empheq}[left=\empheqlbrace]{align}
& \partial_t \rho + \partial_x m =0, \label{continuityEL_baro}  \\
& \partial_t m+ \partial_x \left( \frac{m^2}{\rho} \right) + \partial_x \rho^{\gamma} + m= 0, \label{fluidEL_baro} 
\end{empheq} 
\end{subequations}
with $\rho(0,x)=\rho_0(x) \geq 0$ and $m(0,x)=m_0(x)$ for $x \in \R$, $t \geq 0$, and the adiabatic gas exponent $\gamma >1$. Such a system appears in the mathematical modeling of compressible flow through a porous medium, and its study has drawn a lot of attention over the last decades \cite{marcati2000}. The global existence of $L^{\infty}$ weak entropy solutions to the Cauchy problem of \eqref{EL_baro} is now well established (see for instance \cite{nishida1978} and \cite{huang2006}), and it is known that their long-time asymptotic behavior \cite{huang2000} is governed by the limit diffusive profile:
\begin{subequations} 
\begin{empheq}[left=\empheqlbrace]{align}
& \partial_t \overline{\rho} = \partial_x^2 \overline{\rho}^{\gamma}, \label{porous_medium_eq} \\
& \overline{m}=-\partial_x \overline{\rho}^{\gamma}, \label{darcy_law}
\end{empheq} 
\end{subequations}
where equation \eqref{porous_medium_eq} is the porous media equation, whose fundamental solutions are called Barenblatt solutions \cite{barenblatt1953}, and equation \eqref{darcy_law} is the famous Darcy law. In \cite{liu1996}, the author constructs a class of particular solutions for \eqref{EL_baro} which tend to the Barenblatt solutions $\overline{\rho}$ asymptotically in time, with an explicit rate $\log (t)/t$. In \cite{marcati2005}, the authors find the explicit rates
\[ \|    \rho - \overline{\rho} \|_{L^2(\R)}^2 \leq C (1+t)^{-k_1+\eps}  \ \ \ \text{if} \  1< \gamma \leq2,\]
and 
\[ \|    \rho - \overline{\rho} \|_{L^{\gamma}(\R)}^{\gamma} \leq C (1+t)^{-k_2+\eps}  \ \ \ \text{if} \  \gamma >2,  \]
for any $\eps >0$, with $k_1= \min \left( \frac{\gamma^2}{(\gamma +1)^2}, \frac{\gamma-1}{\gamma} \right)$ and $k_2= \min \left( \frac{\gamma^2}{(\gamma +1)^2}, \frac{1}{\gamma} \right)$, for every $L^{\infty}$ weak entropy solution $(\rho,m)$ of \eqref{EL_baro}. Unfortunately, the decay rates are not in $L^1$-norm, which is the natural norm as \eqref{continuityEL_baro}  and \eqref{porous_medium_eq} both satisfies the conservation of norm 
\[ \| \rho(t,.) \|_{L^1(\R)} =  \|\rho_0 \|_{L^1(\R)} \ \ \ \text{and} \ \ \  \| \overline{\rho}(t,.) \|_{L^1(\R)} =  \|\overline{\rho}_0 \|_{L^1(\R)} \ \ \ \text{for all} \ t \geq 0. \]
Decay rates in $L^1$-norm were achieved for a particular range of $\gamma$ in \cite{huang2011}, where the authors show that
\[ \|    \rho - \overline{\rho} \|_{L^1(\R)} \leq C (1+t)^{-\frac{1}{4(\gamma+1)}+\eps}  \ \ \ \text{if} \  1 < \gamma < 3, \]
which was recently improved and extended in \cite{huang2019} with the estimate
\[ \|    \rho - \overline{\rho} \|_{L^1(\R)} \leq C (1+t)^{-\frac{1}{(\gamma+1)^2}+\eps}  \ \ \ \text{for all} \ \gamma> 1. \]
Throughout the years, a lot of effort has been put to extend the range of $\gamma$ to finally achieve the full range $\gamma > 1$ of barotropic pressure laws $P(\rho)=\rho^{\gamma}$. The next natural step is the study of the case $\gamma=1$, which leads to the isothermal pressure law $P(\rho)=\rho$. However, in the compressible fluid literature, much fewer results are known for the isothermal Euler system with damping
\begin{subequations} \label{EL}
\begin{empheq}[left=\empheqlbrace]{align}
& \partial_t \rho + \partial_x m =0, \label{continuityEL}  \\
& \partial_t m+ \partial_x \left( \frac{m^2}{\rho} \right) + \partial_x \rho + m= 0, \label{fluidEL} 
\end{empheq} 
\end{subequations}
which stands as the limit $\gamma \rightarrow 1$ of the isentropic system \eqref{EL_baro}. In \cite{huang2006}, the authors show the existence of $L^{\infty}$ entropy weak solutions to the Cauchy problem of \eqref{EL}. They also prove, up to a scaling in space $z=x/\sqrt{1+t}$, the $L^p_{\loc}$ convergence of the density $\rho$ towards a diffusive profile $\overline{\rho}$ which satisfies the heat equation
\begin{equation} \label{heat_equation}
\partial_t \overline{\rho} = \partial_x^2 \overline{\rho}, 
    \end{equation}  
in the case $\max(\overline{\rho}_+,\overline{\rho}_-)>0$, where $\overline{\rho}(\pm \infty) = \overline{\rho}_{\pm}$. However, time dependent Gaussian functions, which stand as particular solutions of the heat equation on the whole space, do not satisfy this condition. In a bounded domain $\Omega \subset \R^{d}$, the author shows in \cite{zhao2010} the exponential convergence of any global solution $(\rho,m)$ with small initial data $(\rho_0,m_0)$ towards the limit profile $(\| \rho_0 \|_{L^1}/|\Omega|,0)$. In fact, at least formally, equation \eqref{EL} satisfies the energy inequality
\begin{equation} \label{eq_energy}
    \frac{1}{2} \int_{\R} \frac{m^2}{\rho} + \int_{\R} \rho \log \rho + \int_0^t \int_{\R} \frac{m^2}{\rho} \leq E_0 \ \ \ \text{for all} \ t \geq 0,
\end{equation}
but the left-hand side has no definite sign due to the logarithmic contribution in the potential energy, a property that causes many technical difficulties such as a lack of compactness on the whole space $\R$ when the density $t \mapsto \rho(t,.) \in L^1(\R)$ vanishes at infinity. \\
In \cite{carles_rigidity}, in the case of the compressible Euler equations without frictional damping ($\mu =0$), the authors show that Gaussian functions stand as particular solutions of their system. They also prove that every global weak solution disperses and converges to a universal asymptotic Gaussian profile, up to a rescaling by their dispersion rate, in the weak $L^1$ topology, under some integrability assumptions on the moment of order 2 of the initial data $\rho_0$. We will show in this paper that this kind of property still holds in the case with damping \eqref{EL}, and we will study the asymptotic behavior of Gaussian solutions and moments of order 1 and 2 of the density $\rho$ as they stand as a key ingredient for the proof of this kind of feature. \\

The formal limit  $ \gamma \rightarrow 1$ is singular regarding many features of the isentropic case, especially for the energy inequality of system \eqref{EL_baro}: 
\begin{equation} \label{eq_energy_baro}
    \frac{1}{2} \int_{\R} \frac{m^2}{\rho} + \frac{1}{\gamma-1}\int_{\R} \rho^{\gamma} + \int_0^t \int_{\R} \frac{m^2}{\rho} \leq E_0.
\end{equation}
In this paper we propose to give sense to the formal isothermal limit $\gamma \rightarrow 1$, in particular we will rigorously show and illustrate the convergence of the Barenblatt solutions of \eqref{porous_medium_eq} to the Gaussian solutions of \eqref{heat_equation}, a feature which does not seem to be so much known in the community to the best of the author knowledge.\\

From now on, $(\rho_{\gamma},m_{\gamma})$ with $\gamma > 1$ will denote a solution of the isentropic Euler system \eqref{EL_baro}, whereas $(\rho,m)$ will denote a solution of the isothermal Euler system. The overline bar will be use for solutions of the different limit diffusive profiles. The notation $C$ will denote a generic constant $C > 0$.\\

This paper is organized as follows. In Section 2, we provide energy estimates, assumptions about existence and regularity of solutions of equations \eqref{EL}, and we state the main results of this paper. In Section 3, we show that Barenblatt solutions converge to a Gaussian profile as $\gamma \rightarrow 1$. In Section 4, we explicitly compute the behavior of Gaussian solutions of \eqref{EL}. In Section 5, we study the evolution of the first and second moments of every weak solution of \eqref{EL}. In Section 6, we prove the weak $L^1$ convergence of every solution towards a universal Gaussian profile.  Finally, in Section 7, we give some perspectives about the open question of explicit convergence rate in the isothermal case.

\section{Assumptions and main results}
We first give a notion of global weak solution for the 1-dimensional isothermal Euler equations with damping, and we will assume the global existence of this kind of solutions through the rest of the paper:

\begin{definition} \label{weak_solution_def}
We say that $(\rho,m)$ is a \textbf{weak solution} of system \eqref{EL} in $\left[0,T \right[$ with initial data $(\rho_0,m_0) \in L^1(\R) \times  L^1(\R)$, if there exists locally integrable functions $\sqrt{\rho}$, $\Lambda$ such that, by defining $\rho := \sqrt{\rho}^2$ and $m=\sqrt{\rho} \Lambda$, the following holds:

\begin{enumerate}[label=(\roman*)]

    \item The global regularity:
    \[ \sqrt{\rho} \in L^{\infty}(\left[ 0,T \right[;L^2(\R)), \ \ \ 
      \Lambda  \in L^{2}(\left[ 0,T \right[;L^2(\R)),  \]
      with the compatibility condition
      \[ \sqrt{\rho} \geq 0 \text{ a.e. on } (0,\infty) \times \R, \ \ \ 
      \Lambda =0 \text{ a.e. on } \left\{ \rho=0 \right\}. \]
      
      \item For any test function $\eta \in \mathcal{C}_0^{\infty}(\left[ 0,T \right[ \times \R)$, 
      
      \[ \int_0^T \int_{\R} (\rho \partial_t \eta + m  \partial_x \eta )dx dt + \int_{\R} \rho_0 \eta(0) dx =0, \]
      and for any test function $\zeta \in \mathcal{C}_0^{\infty}(\left[ 0,T \right[ \times \R; \R)$,
      \begin{gather*} \int_0^T \int_{\R} \left( m  \partial_t \zeta + \Lambda^2 \partial_x \zeta + \rho \partial_x \zeta - m \zeta \right) dx dt \\ 
      +  \int_{\R} m_0 \zeta(0) dx =0. \end{gather*}
\end{enumerate}
\end{definition}

\begin{assumption} \label{weak_solution_assumption}
Let $(\rho_0,m_0) \in L^1(\R) \times  L^1(\R)$. We assume that there exists a weak solution $(\rho,m)$ of system \eqref{EL} with initial data $(\rho_0,m_0)$ in the sense of Definition \ref{weak_solution_def} which satisfies, for all $t \geq 0$, the energy estimate
\begin{equation*} \label{energy_EL_eq}
    \frac{1}{2} \int_{\R} \frac{m^2(t,x)}{\rho(t,x)} dx + \int_{\R} \rho(t,x) \log(\rho(t,x)) dx + \int_0^{t} \int_{\R} \frac{m^2(s,x)}{\rho(s,x)} dx ds \leq C,
\end{equation*}
with the additional regularity
\[  (t,x) \mapsto x^2\rho(t,x) \in L^{\infty}( \left[ 0,T \right[; L^1(\R)).\]
\end{assumption}

\begin{remark}
The $L^{\infty}$ entropy weak solutions of system \eqref{EL} described in \cite{huang2006} are in fact weak solutions of \eqref{EL} when the initial data $(\rho_0,m_0) \in L^{\infty}(\R)  \times L^{\infty}(\R)$ satisfies the estimates
\[  0 \leq \rho_0 \leq C \ \ \ \text{and} \ \ \ |m_0| \leq C \rho_0 | \log \rho_0|    \]
(see \cite[Definition 1]{huang2006} for a precise definition of $L^{\infty}$ entropy weak solutions of system \eqref{EL}). Note that the existence of $L^{\infty}$ entropy weak solutions of the isothermal Euler system without damping is proved in \cite{lefloch2005} under the same initial data conditions.
\end{remark}

We now have to introduce several quantities in order to state our results. Let $\tau $ be the unique $\mathcal{C}^{\infty}(\left[ 0,\infty \right))$ solution of the differential equation
\begin{equation*}
\ddot{\tau} = \frac{ 2 }{\tau} - \dot{\tau}, \ \ \ \tau(0)=1, \ \ \ \dot{\tau}(0)=0.
\end{equation*}
From \cite{chauleur_temps_long}, we know that this function satisfies, as $t \rightarrow +\infty$,
\[ \tau(t) \sim 2 \sqrt{t} \ \ \ \text{and} \ \ \ \dot{\tau}(t) \sim \frac{1}{ \sqrt{t}} .   \]
We also introduce the Gaussian function
\[ \Gamma := e^{-y^2}. \]
We make the change of variable $y= x/ \tau(t)$,
\[ \rho (t,x) = \frac{1}{\tau(t)} R \left(t, \frac{x}{\tau(t)} \right)  \ \ \ \text{and } \ \ \  m(t,x) = \frac{1}{\tau(t)^2} M \left(t, \frac{x}{\tau(t)} \right)+ \frac{\dot{\tau}(t)}{\tau(t)} y R \left(t, \frac{x}{\tau(t)} \right) .   \]
We first remark that this change of variable preserves the $L^1$-norm, so for all $t \geq 0$, at least formally 
\[ \int_{\R} R(t,y) dy = \int_{\R} \rho(t,x) dx = \int_{\R} \rho_0(x) dx .  \]

System \eqref{EL} becomes, in the terms of the new unknown $(R,M)$,
\begin{subequations} \label{hydro_renormalized}
\begin{empheq}[left=\empheqlbrace]{align}
& \partial_t R + \frac{1}{\tau^2}\partial_y M =0, \label{continuity_renormalized}  \\
& \partial_t M + \frac{1}{\tau^2} \partial_y \left( \frac{M^2}{R} \right) +\partial_y R +2 y R +  M= 0, \label{fluid_renormalized}
\end{empheq} 
\end{subequations}

which have the following energy inequality:
\begin{equation} \mathcal{E}(t)  +   \int_0^t \frac{\dot{\tau}(t)}{\tau^3(t)} \int_{\R} \frac{M^2}{R} dyds \leq C, \label{scaled_eq_energy} \end{equation}
where
\begin{equation}
   \mathcal{E}(t) =\frac{1}{2 \tau(t)^2} \int_{\R}  \frac{M^2}{R}  dy + \int_{\R}  R\log \left( \frac{R}{\Gamma}  \right) dy. \label{scaled_energy}
\end{equation}
In fact, differentiating with respect to time the left part of \eqref{scaled_eq_energy}, using equations \eqref{continuity_renormalized} and \eqref{fluid_renormalized} and by integration by parts, we get that
\[  \frac{d}{dt} \left(  \frac{1}{2\tau(t)^2} \int_{\R} \frac{M^2}{R} + \int_{\R} R \log R + \int_{\R} y^2 R + \int_0^t \frac{\dot{\tau}(s)}{\tau(t)^3} \int_{\R} \frac{M^2}{R} ds   \right) = -\frac{1}{\tau^2} \int_{\R} \frac{M^2}{R} \leq 0.   \]

Note that from the Csiszár-Kullback inequality (see e.g. \cite{ineg_sobo_log})
\[   \int_{\R} R \log \left( \frac{R}{\Gamma}\right)  \geq \frac{1}{2 \| R_0 \|_{L^1}  }\|R-\Gamma \|^2_{L^1} \geq 0, \]
we direclty get that for all $t \geq 0$, $\mathcal{E}(t) \geq 0$. \\

Denoting by $\mathcal{E}_{\kin}$ the kinetic energy
\[ \mathcal{E}_{\kin}(t) :=\frac{1}{\tau^2(t)} \int_{\R} \frac{M^2(t,y)}{R(t,y)} dy  \]
for all $t\geq 0$, the asymptotics of $\tau$ and $\dot{\tau}$ and the following Lemma  \ref{lemma_entropy} give that
\[ \mathcal{E}_{\kin} \in L^{\infty}(\R_+) \ \ \ \text{and} \ \ \ \int_0^{\infty} \frac{\mathcal{E}_{\kin}(t)}{1+t} dt < \infty,  \]
so we know there exists a sequence $(t_n)_n$, such that $t_n \rightarrow \infty$ and
\[ \mathcal{E}_{\kin}(t_n) \rightarrow 0 \ \ \ \text{as} \ n \rightarrow \infty.   \]
Unfortunately, we have no proof that this property is actually satisfied uniformly in time, even if it seems to be a reasonable assumption that we state in the following:

\begin{assumption} \label{f_to_zero_assump}
We assume that
\[ \mathcal{E}_{\kin}(t) =\frac{1}{\tau^2(t)} \int_{\R} \frac{M^2(t,y)}{R(t,y)} dy \rightarrow 0  \ \ \ \text{as} \  t \rightarrow \infty. \]
\end{assumption}

This assumption is for instance satisfied for Gaussians solutions of system \eqref{EL}, with an explicit convergence rate in $t^{-1/2}$ (see Remark \ref{gaussian_second_moment_remark} below). We now state the two main results of this paper:

\begin{proposition} \label{convergence_momentum_prop}
Under Assumption \ref{weak_solution_assumption} and Assumption \ref{f_to_zero_assump}, we have 
\[  \int_{\R} \left( \begin{array}{c} 1 \\ y \\ y^2 \end{array} \right) R(t,y) dy \longrightarrow \int_{\R} \left( \begin{array}{c} 1 \\ y \\ y^2 \end{array} \right) \Gamma(y) dy \ \ \ \text{as} \ t \rightarrow \infty.   \]
\end{proposition}

\begin{proposition} \label{L1_weak_convergence_prop}
Under Assumption \ref{weak_solution_assumption}, we have
\[ R(t) \underset{t \to \infty}{\rightharpoonup} \Gamma  \text{ weakly in } L^1(\R). \]
\end{proposition}

\begin{remark}
In Proposition \ref{convergence_momentum_prop}, Assumption \ref{f_to_zero_assump} is only required to show the convergence of the moment of order 2. Without this assumption, we are only able to show that the moment of order 2 is uniformly bounded (see Lemma \ref{lemma_entropy}).
\end{remark}

\begin{remark}
Note that Propositions \ref{convergence_momentum_prop} and \ref{L1_weak_convergence_prop} are also satisfied in higher space dimensions $\R^d$ with $d \geq 1$, in the sense that
\[  \int_{\R^d} \left( \begin{array}{c} 1 \\ y \\ |y|^2 \end{array} \right) R(t,y) dy \longrightarrow \int_{\R^d} \left( \begin{array}{c} 1 \\ y \\ |y|^2 \end{array} \right) \Gamma(y) dy \ \ \ \text{as} \ t \rightarrow \infty   \]
with $\Gamma = e^{-|y|^2}$, and
\[ R(t) \underset{t \to \infty}{\rightharpoonup} \Gamma  \text{ weakly in } L^1(\R^d). \]
Of course, our notion of global weak solution has to be adapted to the $d$-dimensional case, as well as system \eqref{EL}, and we refer to \cite{carles_rigidity} and \cite{chauleur_temps_long} for some $d$-dimensional analogue of Definition \ref{weak_solution_def}.\\
In the same vein, the following discussion about the particular Gaussian solutions of \eqref{EL} can easily be generalized to any dimension $d$ by a tensorization property. We also refer the reader to \cite{carles_rigidity} and \cite{chauleur_temps_long} for some $d$-dimensional analogue of system \eqref{system_gaussian_eq_1}-\eqref{system_gaussian_eq_2}-\eqref{system_gaussian_eq_3}.
\end{remark}

\section{The limit $\gamma \rightarrow 1$ of Barenblatt's solutions}
We consider the unique fundamental solution of the porous media equation (with the Dirac delta function as initial data)
\begin{empheq}[left=\empheqlbrace]{equation}
    \begin{aligned}
    & \partial_t \overline{\rho}_{\gamma} = \partial_x^2  \left[ \left( \overline{\rho}_{\gamma}\right)^{\gamma} \right], \\
     & \overline{\rho}_{\gamma}(-1,x) = \lambda \delta(x), \ \ \ \lambda>0,
   \end{aligned}
\end{empheq}
which is called the Barenblatt solution \cite{aronson1986}. Note that we take the initial data at $t=-1$ to avoid the singularity at $t=0$. We recall that this solution can be written 
\begin{equation} \label{barenblatt_explicit_eq}
 \overline{\rho}_{\gamma}(t,x) = (1+t)^{-\frac{1}{\gamma+1}} \left[ A-B \xi^2 \right]_+^{\frac{1}{\gamma-1}}   
 \end{equation}
with $\xi = x(1+t)^{-\frac{1}{\gamma+1}}$, $\left[ f \right]_+ = \max \left\{f,0 \right\}$, $B=\frac{\gamma-1}{2 \gamma (\gamma+1)}$ and A determined by
\begin{equation} \label{expression_A}
 2 A^{\frac{\gamma+1}{2(\gamma-1)}} B^{-\frac{1}{2}} \int_0^{\frac{\pi}{2}} (\cos \theta )^{\frac{\gamma+1}{\gamma-1}} d \theta = \lambda.    
 \end{equation} 
Moreover, $\overline{\rho}_{\gamma}$ is continuous on $\R$, it satisfies the conservation of mass
\[ \int_{\R} \overline{\rho}_{\gamma}(t,x) dx = \lambda\]
for all $t \geq -1$, and has compact support for any finite time $T>0$, namely
\[ \overline{\rho}_{\gamma}=0 \ \ \ \text{if} \ \ \ |\xi| \geq \sqrt{A/B}.      \]
The power $\frac{1}{\gamma-1}$ in \eqref{barenblatt_explicit_eq} could make the limit $\gamma \rightarrow 1$ unclear, however we can show:

\begin{proposition}
We make the change of variable $\xi= x(1+t)^{-\frac{1}{\gamma+1}}$ and
\[ \overline{\rho}_{\gamma} (t,x) = (1+t)^{-\frac{1}{\gamma+1}} \mathcal{B}_{\gamma} \left(x(1+t)^{-\frac{1}{\gamma+1}} \right) \]
which preserves the $L^1$-norm. Then,
\[ \mathcal{B}_{\gamma} \rightarrow \frac{\lambda}{2\sqrt{\pi}} e^{-\frac{\xi^2}{4}} \ \ \ \text{in} \ L^{\infty}(\R) \ \text{as} \ \gamma \rightarrow 1 .\]
In particular, the $L^{\infty}$ estimate
\[ |\overline{\rho}_{\gamma}(t,x)| \leq A^{\frac{1}{\gamma -1}} (1+t)^{-\frac{1}{\gamma+1}}    \]
has the following continuous limit when $\gamma \rightarrow 1$:
\[ |\overline{\rho}(t,x)| \leq \frac{\lambda}{2\sqrt{\pi}} (1+t)^{-\frac{1}{2}},    \]
where $\overline{\rho}$ denotes the Gaussian limit of the Barenblatt profile $\overline{\rho}_{\gamma}$ as $\gamma \rightarrow 1$ (namely $\overline{\rho}:=\overline{\rho}_1$, we omit the index $\gamma$ when $\gamma =1$).
\end{proposition}
\begin{proof}
We have,  for $|\xi| <\sqrt{A/B}$,
\begin{equation*}
    \mathcal{B}_{\gamma}(\xi) = \left[ A - B \xi^2   \right]_+^{\frac{1}{\gamma-1}} 
    = \exp \left(  \frac{1}{\gamma-1} \log A + \frac{1}{\gamma-1} \log \left( 1 - \frac{B}{A} \xi^2 \right)   \right),
\end{equation*}
as $A >0$ from expression \eqref{expression_A} and $\frac{B}{A} \xi^2 < 1$. We have now two terms to handle when $\gamma \rightarrow 1$. Taking the logarithm in equation \eqref{expression_A}, and recalling that $B=\frac{\gamma-1}{2 \gamma (\gamma+1)}$, we have
\[ \frac{1}{\gamma-1} \log A = \frac{2}{\gamma+1} \left( \log \left( \frac{\lambda}{2 \sqrt{2 \gamma (\gamma+1)}} \right) -  \log \left( \frac{1}{\sqrt{\gamma-1}}   \int_0^{\frac{\pi}{2}} (\cos \theta )^{\frac{\gamma+1}{\gamma-1}} d \theta \right)\right).    \]
We denote $\eps = \gamma -1$, so we look at the limit $\eps \rightarrow 0$ of the following Wallis integral
\[  W_{1+\frac{2}{\eps}} = \int_0^{\frac{\pi}{2}} (\cos \theta)^{1+\frac{2}{\eps}} d \theta = \int_0^{\frac{\pi}{2}} (\cos \theta )^{\frac{\gamma+1}{\gamma-1}} d \theta. \]
We recall the well-known equivalent of Wallis integral when $\eps \rightarrow 0$,
\[ W_{1+\frac{2}{\eps}} = \sqrt{\frac{\pi}{2\left( 1+ \frac{2}{\eps} \right)}} + O \left( \frac{1}{1+ \frac{2}{\eps}} \right) =  \sqrt{\frac{\pi \eps}{4+2\eps}} + O(\eps),  \]
so it is easy to check that
\[ \frac{1}{\sqrt{\gamma-1}}   \int_0^{\frac{\pi}{2}} (\cos \theta )^{\frac{\gamma+1}{\gamma-1}} d \theta =  \sqrt{\frac{\pi}{4}} \left( 1+O(\eps)  \right)\]
as $\gamma \rightarrow 1$, and then
\[ \exp \left( \frac{1}{\gamma-1} \log A \right) = \frac{\lambda}{2\sqrt{\pi}} \frac{1}{1+O(\eps)}.   \]
For the second term, as we have 
\[ \frac{B}{A} = B \left( \frac{2 W_{1+\frac{2}{\eps}}}{\lambda B^{\frac{1}{2}}}  \right)^{\frac{2 \eps}{2+\eps}} = \frac{\eps}{4} + O(\eps^2)   \]
from the previous equivalent, we write
\[ \frac{1}{\eps} \log \left( 1 - \frac{B}{A} \xi^2 \right) = - \frac{1}{\eps} \sum_{n \geq 1} \frac{1}{n} \left( \frac{B}{A} \xi^2    \right)^n \sim -\frac{1}{4} \xi^2 + \eps \sum_{n \geq 2}  \frac{1}{n} \left( \frac{\xi^2}{4}    \right)^n \eps^{n-2}, \]
so finally 
\[ \exp \left( \frac{1}{\gamma-1} \log \left( 1 - \frac{B}{A} \xi^2 \right) \right) \longrightarrow e^{-\frac{\xi^2}{4}},   \]
and 
\[  \mathcal{B}_{\gamma} \left( \xi \right) - \frac{\lambda}{2\sqrt{\pi}} e^{-\frac{\xi^2}{4}} = \frac{\lambda}{2\sqrt{\pi}} e^{-\frac{\xi^2}{4}} \left( \frac{1}{1+O(\eps)} e^{O(\eps)} -1   \right) = O(\eps). \]
On $|\xi| \geq \sqrt{A/B}$, we know that $ \mathcal{B}_{\gamma} \left( \xi \right)=0$, and as $\xi \rightarrow e^{- \xi^2/4}$ is a Gaussian, 
\[    \sup_{|\xi| \geq \sqrt{A/B}} \left|  \mathcal{B}_{\gamma} \left( \xi \right) - \frac{\lambda}{2\sqrt{\pi}} e^{-\frac{\xi^2}{4}}  \right| = \frac{\lambda}{2\sqrt{\pi}} e^{-\frac{A}{B}}  \leq C e^{-\frac{1}{\eps}},\]
so finally
\[  \sup_{\xi \in \R} \left|  \mathcal{B}_{\gamma} \left( \xi \right) - \frac{\lambda}{2\sqrt{\pi}} e^{-\frac{\xi^2}{4}}  \right| \leq C e^{-\frac{1}{\gamma -1}} + O(\gamma-1) \longrightarrow 0 \ \text{as} \ \gamma \rightarrow 1,   \]
which ends the proof.
\end{proof}

In order to illustrate this result, we make the following numerical simulation of the function $\mathcal{B}_{\gamma}$ for several values of $\gamma$ (namely $\gamma=2$, $1.5$ and $1.1$), and the initial condition constant $\lambda=1$. We also plot the limit Gaussian profile $e^{-\xi^2/4}/2\sqrt{\pi}$, which corresponds to the case $\gamma=1$.

\begin{figure}[h]
	\centering
		\includegraphics[width=0.85\textwidth, clip]{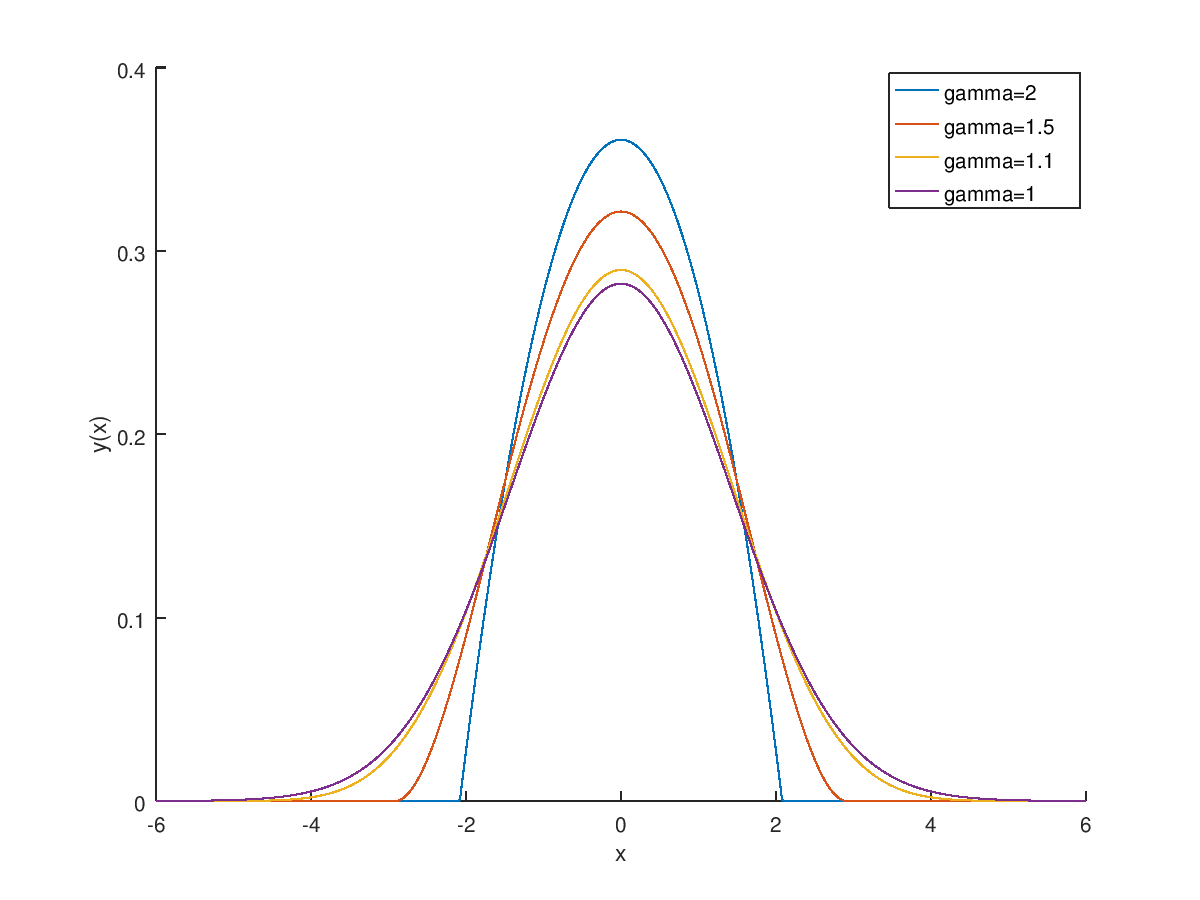}	
	\caption{Convergence of $\mathcal{B}_{\gamma}$ towards its limit Gaussian profile.}
	\label{fig:barenblatt_to_gaussian}
\end{figure}

\section{Gaussian solutions}
In this section, following \cite{carles_rigidity}, we seek for particular Gaussian solutions of \eqref{EL} of the form
\[ \overline{\rho}(t,x)= b(t) e^{-\alpha(t) (x-\overline{x}(t))^2}  \]
and 
\[ \overline{m}(t,x)= (\beta(t) x + c(t) ) \overline{\rho}(t,x),    \]
with the initial conditions $b(0)=b_0 >0$, $\alpha(0)=\alpha_0 >0$, $\beta(0)=\beta_0 \in \R $, $c(0)=c_0 \in \R$. As system \eqref{EL} is invariant by translation, we assume that $\overline{x}(0)=0$. We also denote $\rho_0=\overline{\rho}(0,.)$ and $m_0=\overline{m}(0,.)$. Plugging these expressions into \eqref{EL}, we obtain the following set of differential equations:
\begin{equation} \label{system_gaussian_eq_1}
    \dot{\alpha}+2 \alpha b=0, \ \ \ \dot{\beta}+\beta^2+\beta=2 \alpha,\\
\end{equation}
\begin{equation}  \label{system_gaussian_eq_2}
    \dot{\overline{x}}=\beta \overline{x} +c, \ \ \  \dot{b}=b(\dot{\alpha} \overline{x}^2 +2 \alpha c\overline{x} (\dot{\overline{x}}-c) -\beta),
\end{equation}
\begin{equation}  \label{system_gaussian_eq_3}
     \dot{c}+\beta c +c = -2 \alpha \overline{x}.
\end{equation}

In order to solve this system, mimicking \cite{li_wang_2006}, we can check that the two equations of \eqref{system_gaussian_eq_1} are satisfied if and only if $\alpha$ and $\beta$ are of the form
\[ \alpha(t)= \frac{\alpha_0}{\tau(t)^2}, \ \ \ \beta(t) = \frac{\dot{\tau}(t)}{\tau(t)}, \]
where  $\tau$ is the global solution of the differential equation
\begin{equation} \label{eq_tau}
\ddot{\tau} = \frac{ 2 \alpha_0 }{\tau} - \dot{\tau}, \ \ \ \tau(0)=1, \ \ \ \dot{\tau}(0)=\beta_0.
\end{equation}
We recall (see \cite{chauleur_temps_long}) that there exists a unique global solution $\tau \in \mathcal{C}^{\infty}(\left[ 0,\infty \right))$ to this nonlinear ODE, and that this solution remains uniformly bounded from below by a strictly positive constant. \\ Plugging these expressions into the second equation of \eqref{system_gaussian_eq_2}, we also get that
\[ b(t) = \frac{b_0}{\tau(t)}. \]
We also get an expression of $c$ in terms of the center $\overline{x}$ of our Gaussian:
\[ c(t)=c_0 -\left(1+\frac{\dot{\tau}}{\tau}\right) \overline{x}.   \]
Let first show that the center of the Gaussian has an explicit expression, which does not depend on the function $\tau$:
\begin{proposition}
We have
\[ \overline{x}(t) = \frac{1}{\| \rho_0\|_{L^1(\R)}} \left( \int_{\R} x \rho_0 - (1-e^{-t}) \int_{\R} m_0    \right) .   \]
In particular, there exists $\overline{x}_{\infty} \in \R$ such that
\[ \overline{x}(t) \rightarrow  \overline{x}_{\infty} \ \ \ \text{as} \ \ \ t \rightarrow \infty.  \]
\end{proposition}
\begin{proof}
We know that $\overline{\rho}$ is a Gaussian function centered in $\overline{x}$, so we get that for all $t \geq 0$, 
\[ \int_{\R} (x - \overline{x}(t)) \overline{\rho}(t,x) dx =0, \]
hence we only have to study the first moment of $\overline{\rho}$ keeping in mind that \[\overline{x}(t) = \frac{1}{ \| \rho_0 \|_{L^1(\R)}} \int_{\R} x \overline{\rho}(t,x) dx. \]
Integrating equation \eqref{fluidEL} over $\R$ we get that
\[ \frac{d}{dt} \left(\int_{\R}  \overline{m}(t,x) dx \right)= \int_{\R}  \partial_t \overline{m}(t,x) dx = -  \int_{\R}  \overline{m}(t,x) dx,  \]
so we get by integration by parts that
\begin{equation*} 
\frac{d}{dt} \left(\int_{\R} x \overline{\rho}(t,x) dx \right)= \int_{\R} x  \partial_t \overline{\rho}(t,x) dx = - \int_{\R} \overline{m}(t,x) dx = - e^{- t} \int_{\R} m_0(x) dx,
\end{equation*} 
so finally integrating this expression over $\left[0, t \right]$ we get the result.
\end{proof}

In order to know the behavior of our Gaussian functions when $t \rightarrow \infty$, we only need to get some equivalents of the real function $\tau$, which is the goal of the following proposition. We refer to \cite{chauleur_temps_long} for a complete study of the differential equation \eqref{eq_tau}.

\begin{proposition} \label{center_gaussian}
Let  $\tau$ be the unique $\mathcal{C}^{\infty}(\left[ 0,\infty \right))$ solution of the differential equation \eqref{eq_tau}, then we have
\[ \tau(t) \sim 2 \sqrt{\alpha_0 t} \ \ \ \text{and} \ \ \ \dot{\tau}(t) \sim \sqrt{\frac{\alpha_0}{ t}}.    \]
In particular,
\[ \alpha(t) \sim \frac{1}{4 t}, \ \ \ \beta(t) \sim \frac{1}{2 t}, \ \ \    b(t) \sim \frac{b_0}{2 \sqrt{\alpha_0 t}} \ \ \ \text{and} \ \ \ c(t) \rightarrow c_0-\overline{x}_{\infty}. \]
\end{proposition}

\section{Evolution of certain quantities}
In this section, we are going to give two propositions that describe the behavior of respectively the first and second moment of the renormalized density $R$.
\begin{proposition}
We denote 
\[ I_1 = \int_{\R} M dy  \ \ \ \text{and } \ \ \ I_2 = \int_{\R} y R dy. \]
Then, there exists $C \geq 0$ such that 
\[ I_1(t) =e^{-t} I_1(0) \ \ \ \text{and} \ \ \ I_2(t) \sim \frac{C}{\sqrt{t}} \ \ \ \text{as} \ t \rightarrow \infty.  \]
\end{proposition}
\begin{proof}
First off, mimicking the calculus of the proof of Proposition \ref{center_gaussian}, we have by integrating equation \eqref{fluidEL} over $\R$ and integration by parts that
\[  \frac{d}{dt} \left( \int_{\R} M dy \right) =  - \int_{\R} M dy ,  \]
so we easily get that
\[  I_1(t) =e^{-t} I_1(0).  \]
Now, by integration by parts, and using \eqref{hydro_renormalized}, we get the system of coupled differential equations:
\[ \dot{I}_1 = \int_{\R} \partial_tM = - \int_{\R} \left[ \frac{1}{\tau^2} \partial_x \left( \frac{M^2}{R} \right) + \partial_y R + 2 y R + M \right] =- I_1 - 2 I_2,  \]
and 
\[  \dot{I}_2 =  \int_{\R} y \partial_t R = - \int_{\R} y \frac{\partial_x M}{\tau^2} = \frac{1}{\tau^2} I_1. \]
We denote $\tilde{I}_2 = \tau I_2$, so that
\[ \dot{\tilde{I}}_2= \dot{\tau} I_2 + \tau \dot{I}_2 = \dot{\tau} I_2 + \frac{1}{\tau} I_1,   \]
and
\[ \ddot{\tilde{I}}_2  =  \ddot{\tau} I_2 + \dot{\tau} \dot{I}_2 - \frac{\dot{\tau}}{\tau} I_1 + \frac{1}{\tau} \dot{I}_1  =  - (\dot{\tau} I_2 + \frac{1}{\tau} I_1 ) = - \dot{\tilde{I}}_2,   \]
hence
\[ \dot{\tilde{I}}_2(t) = \dot{\tilde{I}}_2(0) e^{- t} \]
and 
\[ \tilde{I}_2(t)= \tilde{I}_2(0) + \dot{\tilde{I}}_2(0) ( e^{-t}-1).   \]
We easily compute the initial conditions
\[  \dot{\tilde{I}}_2(0) = I_1(0) \ \ \ \text{and} \ \ \ \tilde{I}_2(0)= I_2(0), \]
so finally
\[I_2(t) = \frac{1}{\tau(t)} \left[ I_2(0) - I_1(0) (1-e^{-t}) \right] \sim  \frac{1}{\tau(t)}  \left[ I_2(0) - I_1(0) \right] \]
when the initial condition are not well prepared in the sense that $I_2(0) \neq I_1(0)$. Note that if  $I_2(0) = I_1(0)$, we have
\[ I_2(t)= \frac{e^{-t}}{\tau(t)} I_1(0).   \]
\end{proof}

We now give an useful lemma, which induces the boundedness of the second moment of the density $R$ uniformly with respect to time:

\begin{lemma} \label{lemma_entropy}
There holds
\begin{equation} \label{equation_lemma_1}
    \sup_{t\geq 0}  \int_{\R}  R(t,y)(1 +y^2 +|\log R(t,y) |)  dy  < \infty
\end{equation}
and
\begin{equation} \label{equation_lemma_2}
    \int_0^{\infty} \frac{\dot{\tau}(t)}{\tau^3(t)}\int_{\R}  \frac{M^2}{R}  dy dt  < \infty.
\end{equation}
\end{lemma}
\begin{proof}
Since $\mathcal{E}(t) \geq 0$, \eqref{equation_lemma_2} follows from \eqref{scaled_eq_energy}. We define
\[ \mathcal{E}_+ := \frac{1}{2 \tau^2} \int_{\R}  \frac{M^2}{R}  dy +  \int_{R > 1} R \log R +  \int_{\R} y^2 R,  \]
such that Equation \eqref{scaled_eq_energy} gives
\[ \mathcal{E}_+(t) +   \int_0^t \frac{\dot{\tau}(t)}{\tau^3(t)} \int_{\R} \frac{M^2}{R} dyds \leq  C+  \int_{R < 1} R \log\left(\frac{1}{R} \right) . \]
The last term is controlled by
\[  \int_{R < 1} R \log\left(\frac{1}{R} \right) \leq C_{\eps}  \int_{\R} R^{1-\eps}  \]
for all $\eps >0$ arbitrary small. By an interpolation formula, we get that
\[ \int_{\R} R^{1-\eps}  \lesssim \| R \|_{L^1}^{1-\frac{3\eps}{2} }  \| y^2 R \|_{L^1}^{ \eps/2 } \]
for all $0< \eps < \frac{2}{3}$. This implies that for all $t \geq 0$, 
\[ \mathcal{E}_+(t) +  \int_0^t\frac{\dot{\tau}(t)}{\tau^3(t)} \int_{\R} \frac{M^2}{R} dyds \lesssim  1 + \mathcal{E}_+^{ \eps/2}(t),  \]
thus $\mathcal{E}_+(t) \in L^{\infty}(\R_+)$, and  equation \eqref{equation_lemma_1} follows. 
\end{proof}

With Assumption \ref{f_to_zero_assump}, we can then show a much better result on the second moment of the density $R$ than its boundedness, which is the convergence of this second moment towards the second moment of the Gaussian $\Gamma$:

\begin{proposition} \label{second_moment_prop}
We denote 
\[ J_1 = \int_{\R} y^2 (R - \Gamma) dy. \]
Then, under Assumption \ref{f_to_zero_assump},
\[ |J_1(t)| \rightarrow 0 \ \ \ \text{as} \ t \rightarrow +\infty.  \]
\end{proposition}
\begin{proof}
Denoting 
 \[ J_2 = \int_{\R} y M dy, \]
we first write the system of differential equations satisfied by $J_1$ and $J_2$. By integration by parts and using equation \eqref{continuity_renormalized}, we compute
\[ \dot{J}_1 =  \frac{d}{dt} \int_{\R} y^2 \partial_t (R- \Gamma) = - \int_{\R} y^2 \frac{\partial_y M}{\tau^2} = \frac{2}{\tau^2} \int_{\R} y M = \frac{2}{\tau^2} J_2.    \]
Still by integration by parts, and using equation \eqref{fluid_renormalized}, we get that
\[  \dot{J}_2 = \int_{\R} y \partial_t M = - \int_{\R} y \left[ \frac{1}{\tau^2} \partial_y \left( \frac{M^2}{R} \right) + \partial_y R + 2 y R + M  \right]  = \frac{1}{\tau^2} \int_{\R} \frac{M^2}{R} + \int_{\R} R - \int_{\R} 2y^2 R - \int_{\R} yM.            \]
Here we introduce the function $\Gamma$, writing
\[ \int_{\R} R -2 \int_{\R} y^2 R = \int_{\R} \Gamma - 2 \int_{\R} y^2 \Gamma  - \int_{\R} 2y^2(R - \Gamma),  \]
and we can remark by an easy calculation that
\[  \int_{\R} y^2 \Gamma = \frac{1}{2} \int_{\R} \Gamma,  \]
so that finally 
\[  \dot{J}_2 + J_2 + 2 J_1 =  \frac{1}{\tau^2} \int_{\R} \frac{M^2}{R}. \]
We denote $\tilde{J}_1 = \tau J_1$, so that
\[ \dot{\tilde{J}}_1= \dot{\tau} J_1 + \tau \dot{J}_1 = \dot{\tau} J_1 + \frac{2}{\tau} J_2,   \]
and
\[ \ddot{\tilde{J}}_1  =  \ddot{\tau} J_1 + \dot{\tau} \dot{J}_1 - \frac{2\dot{\tau}}{\tau} J_2 + \frac{2}{\tau} \dot{J}_2  =  - \dot{\tilde{J}}_1 - \frac{2}{\tau} J_1 + \frac{2}{\tau^3} \int_{\R} \frac{M^2}{R},  \]
so we write
\[  \ddot{\tilde{J}}_1 + \dot{\tilde{J}}_1 +\frac{1}{\tau^2} \tilde{J}_1= \frac{2}{\tau^3} \int_{\R} \frac{M^2}{R}.  \]
Rather than trying to solve directly this non-autonomous differential equation of order 2, we will work on the following approximate equation, for $t \geq 1$, 
\begin{equation} \label{approx_ODE}
    \ddot{f} + \dot{f} + \frac{1}{4t} f= \frac{2}{\tau} \mathcal{E}_{\kin}.
\end{equation}
 In fact, denoting $w= \tilde{J}_1-f$, we see that $w$ satisfies the differential equation
\[ \ddot{w}+ \dot{w} + \frac{1}{4t}w + \left(  \frac{1}{\tau^2} - \frac{1}{4t} \right)\tilde{J}_1 =0,  \]
with
\[  \left(  \frac{1}{\tau^2} - \frac{1}{4t} \right)= \frac{(\tau -2\sqrt{t} ) ( \tau+2\sqrt{t})}{4t \tau^2} = O \left( \frac{1}{t^{\frac{3}{2}}}  \right)   \]
as we actually know from \cite{chauleur_temps_long} that $\tau(t) - 2 \sqrt{t}$ is uniformly bounded. The homogeneous part of equation \eqref{approx_ODE} can be written under a Kummer's type equation
\[ t \ddot{f} + t \dot{f} + \frac{1}{4} f=0   \]
which has two fundamental independent solutions (see \cite{abramowitz64}) 
\[ f_1(t)= t \mathcal{M} \left( \frac{5}{4} , 2 , -t \right)=t e^{-t} \mathcal{M} \left( \frac{3}{4} , 2 , t \right) \]
and
\[    f_2(t) = e^{-t} \mathcal{U} \left(- \frac{1}{4} , 0 , t \right) = t e^{-t} \mathcal{U} \left( \frac{3}{4} , 2 , t \right), \]
where $\mathcal{M}(a,b,z)$ and $\mathcal{U}(a,b,z)$ respectively denotes the Krummer's and the Tricomi's function, which both stand as confluent hypergeometric functions and are independent solutions of the Kummer's equation
\[ z \frac{d^2 w}{d z^2} + (b-z)\frac{d w}{d z} -a w=0. \]
In particular, from the asymptotic properties \cite{abramowitz64}
\[ \mathcal{M}(a,b,z)= \frac{\Gamma(b)}{\Gamma(a)}e^{z} z^{a-b} \left( 1 + O \left( |z|^{-1} \right)    \right)   \]
and
 \[ \mathcal{U}(a,b,z)= z^{-a} \left( 1 + O \left( |z|^{-1} \right)    \right), \]
we get the following asymptotic for our fundamental solutions
\begin{equation} \label{equivalent_y1}
 f_1(t) \sim \frac{C}{t^{\frac{1}{4}}} \left( 1+ O \left( \frac{1}{t} \right) \right) 
 \end{equation}
and
\begin{equation} \label{equivalent_y2}
f_2(t) \sim e^{-t} t^{\frac{1}{4}} \left( 1+ O \left( \frac{1}{t} \right) \right).    
\end{equation}
From the classical theory of linear differential equations, we know that every solution of equation \eqref{approx_ODE} can be written under the form
\begin{equation} \label{general_solution_ODE}
f(t)= c_1 f_1(t) + c_2 f_2(t) + f_0(t),    
\end{equation}
where $c_1$ and $c_2$ denote two real numbers, and $f_0$ is a particular solution of \eqref{approx_ODE}. We easily compute the Wronskian function
\[ W(t) = f_1(t) \dot{f}_2(t) -  \dot{f}_1(t) f_2(t) \]
by solving the differential equation
\[  \dot{W} =\dot{f}_1 \dot{f}_2 + f_1 \ddot{f}_2 - \ddot{f}_1 f_2 - \dot{f}_1 \dot{f}_2 = f_1( -\dot{f}_2 - \frac{1}{\tau^2} f_2 ) + (\dot{f}_1 + \frac{1}{\tau^2} f_1) f_2= -W,  \]
which leads to
\[    W(t)= W_0 e^{-t} .\]
Then, by a variation of constant formula, we can find a particular solution of \eqref{approx_ODE} under the form
\[ f_0(t) = f_2(t) \int_0^t \frac{f_1(s)}{W(s)} \frac{\mathcal{E}_{\kin}(s)}{\tau(s)} ds - f_1(t) \int_0^t  \frac{f_2(s)}{W(s)} \frac{\mathcal{E}_{\kin}(s)}{\tau(s)} ds, \]
so that using Assumption \ref{f_to_zero_assump} and the equivalents \eqref{equivalent_y1} and \eqref{equivalent_y2}, we get that every solution of \eqref{approx_ODE} has the asymptotic:
\[ | f(t) | = o(\sqrt{t}) .   \]
In fact, in the expression of \eqref{general_solution_ODE}, \[   \left| f_1(t) \int_0^t e^s f_2(s) \frac{\mathcal{E}_{\kin}(s)}{\tau(s)} ds \right| \lesssim  \frac{1}{t^{\frac{1}{4}}} \int_1^t s^{-\frac{1}{4}} \mathcal{E}_{\kin}(s) ds = o(\sqrt{t}),  \]
and every other terms is bounded as $t \rightarrow \infty$. The same result applies for the function $w=\tilde{J}_1 - f$, so we can write that for all $t \geq t_0 \geq 0$,
\begin{multline} \label{gronwal_J1_tilde}
\tilde{J}_1(t) = f(t) +c_1 f_1(t) + c_2 f_2(t) + f_2(t) \int_{t_0}^t \frac{f_1(s)}{W(s)} \left( \frac{1}{4s} - \frac{1}{\tau(s)^2} \right) \tilde{J}_1(s) ds \\
- f_1(t) \int_{t_0}^t  \frac{f_2(s)}{W(s)} \left( \frac{1}{4s} - \frac{1}{\tau(s)^2} \right) \tilde{J}_1(s) ds.    
\end{multline} 
We already know from Lemma \ref{lemma_entropy} that $\tilde{J}_1$ is $O(\sqrt{t})$, and as we have just shown that $f$ is $o(\sqrt{t})$, for $\eps >0$ fixed, there exists $t_0 \geq 0$ such that, for all $t \geq t_0$,
\[ | \tilde{J}_1(t)| \leq C \sqrt{t}, \ \ \  \left| \frac{1}{4s} - \frac{1}{\tau(s)^2}  \right|   \leq Ct^{-\frac{3}{2}} \ \ \ \text{and} \ \ \ \frac{1}{\sqrt{t}}\left| f(t)+c_1 f_1(t) + c_2 f_2(t)  \right| \leq \frac{\eps}{2} .   \]
Injecting these inequalities in the right-hand side of \eqref{gronwal_J1_tilde}, we have
\[ \frac{1}{\sqrt{t}} | \tilde{J}_1(t)| \leq  \frac{\eps}{2} + \frac{C}{t^{\frac{1}{4}}} + \frac{C}{t^{\frac{3}{4}}}  \leq \eps  \]
for $t$ large enough, so we finally get by a bootstrap argument that
\[   | \tilde{J}_1(t) |   = o(\sqrt{t}).  \]
Finally, recalling that $\tilde{J}_1 = \tau J_1$ and that $\tau \sim 2 \sqrt{t}$, we can conclude that
\[ |J_1| = \left| \int_{\R} y^2(R-\Gamma) dy  \right| = o(1),  \]
which ends the proof.
\end{proof}

\begin{remark} \label{gaussian_second_moment_remark}
Note that in the Gaussian case, we can explicitly compute the kinetic energy
\[ \mathcal{E}_{\kin}(t)=\frac{1}{\tau^2(t)} \int_{\R} \frac{M^2(t,y)}{R(t,y)} dy = \frac{c(t)}{\tau(t)} \| \rho_0 \|_{L^1(\R)} = O \left( \frac{1}{\sqrt{t}} \right),  \]
unless the initial data are well prepared in the sense that $c_0=\overline{x}_{\infty}$ (which would actually improve this rate). This feature gives an explicit rate of convergence of $\mathcal{E}_{\kin}$ towards 0 that can be propagated to the convergence of the second momentum of $R$ by adapting the proof of Proposition \ref{second_moment_prop}, namely
\[  | J_1 | =   \left| \int_{\R} y^2(R-\Gamma) dy  \right| = O \left( \frac{1}{\sqrt{t}} \right). \]
\end{remark}

\section{Convergence}

\textit{Proof of Proposition \ref{L1_weak_convergence_prop}.}
We are first going to try to eliminate the momentum $M$ of our target equation. Differentiating equation \eqref{fluid_renormalized}  with respect to $y$, and using equation \eqref{continuity_renormalized} in order to express
\[ \partial_t ( \partial_y (RM)) = - \partial( \tau^2 \partial_t R),    \]
we get that
\begin{equation*} 
 - \partial_t ( \tau^2 \partial_tR) -  \tau^2 \partial_t R +  L R =-\frac{1}{\tau^2} \partial_y \left( \frac{M^2}{R} \right) . 
 \end{equation*}
where we have defined the Fokker-Planck operator 
 \[L:=\partial_y^2  +2  \partial_y(y \cdot ).\] 
Following \cite{chauleur_temps_long}, we introduce another scaling in time $s$ defined by
 \begin{equation} \label{scaling}
 \partial_s = \tau^2 \partial_t,
 \end{equation}
 and the notation
 \[ \tilde{R}(s(t),y):= R(t,y). \]
We calculate the quantities
 \[ \partial_t ( \tau^2 \partial_tR) =  \frac{1}{\tau^2} \partial_s^2 \tilde{R} \ \text{and} \   \tau^2 \partial_t R =  \partial_s \tilde{R},  \]
hence we obtain the following equation:
 \begin{equation} \label{equation_s}
    -  \frac{1}{\tau^2} \partial_s^2 \tilde{R} - \partial_s \tilde{R} +  L \tilde{R} = -\frac{1}{\tau^2} \partial_y \left( \frac{M^2}{\tilde{R}} \right) . 
 \end{equation}
Now we remark that equation  \eqref{equation_lemma_1} induces
\begin{equation} \label{scaled_equation_lemma_1}
    \sup_{s\geq 0}  \int_{\R} \tilde{R}(s,y) (1 +|y|^2 +|\log \tilde{R}(s,y) |)  dy   < \infty, 
\end{equation} 
that equation \eqref{equation_lemma_2} gives
\begin{equation} \label{scaled_equation_lemma_2}
\int_0^{\infty}\dot{\tau}(t) \int_{\R} \frac{\tilde{M}^2}{\tilde{R}}  dydt  < \infty,  
\end{equation}
and that 
\[   \tau(s) \sim 2  e^{2s} \ \text{and} \ \dot{\tau}(s)\sim e^{-2s}, \]
so we can conclude like in \cite{carles2018}. Let a sequence $s_n \rightarrow \infty$, take $s\in \left[-1,2 \right]$, and denote
\[ \tilde{R}_n(s,y):=\tilde{R}(s+s_n,y).\]
From \eqref{scaled_equation_lemma_1} along with the de la Vallée-Poussin \cite{dellacherie1975} and Dunford-Pettis theorems \cite{dunford1958}, we get the following weak convergence (up to a subsequence, not relabeled for reader's convenience), for all $p \in \left[1,\infty \right)$,
\[ \tilde{R}_n  \underset{t \to \infty}{\rightharpoonup} \tilde{R}_{\infty} \ \ \ \text{in} \ L^p(-1,2;L^1(\R)). \]
We also get the weak convergence of the initial datum, up to another subsequence:
\[ \tilde{R}_n(0)  \underset{t \to \infty}{\rightharpoonup} \tilde{R}_{0,\infty} \ \ \ \text{in} \ L^1(\R). \]
Thanks to \eqref{scaled_equation_lemma_1}, we also get that the family $( \tilde{R}(s_n,.))_n$ is tight, so
\[ \int_{\R} \tilde{R}_{0,\infty}(y) dy = \int_{\R} \Gamma(y) dy \]
and
\[ \int_{\R} \tilde{R}_{0,\infty}(y)(1 +|y|^2 +|\log  \tilde{R}_{0,\infty}(y) |)   dy   < \infty. \]
Then, denoting $\tau_n(s)=\tau(s+s_n)$, equation \eqref{scaled_equation_lemma_2} implies that
\[ -\frac{1}{\tau_n^2} \partial_y \left( \frac{M_n^2}{\tilde{R}_n} \right)  \underset{t \to \infty}{\rightharpoonup} 0 \ \ \ \text{in} \ L^1(-1,2;W^{-2,1}(\R)). \]
In addition, in \eqref{equation_s}, all the other terms but two obviously go weakly to zero, which yields
\begin{equation} \label{eq_infty}
    \partial_s \tilde{R}_{\infty} = L \tilde{R}_{\infty} 
\end{equation}
in $\mathcal{D}'((-1,2) \times \R)$, with $\tilde{R}_{\infty}(0,;)= \tilde{R}_{0,\infty} \in L^1(\R)$. Thanks to the above bounds on $\tilde{R}_{0,\infty}$, it is known (see \cite{arnold2000}) that the  solution $\tilde{R}_{\infty}$ to \eqref{eq_infty} is actually defined for all $s \geq0$ and satisfies 
\begin{equation} \label{conv_gaussian_infty}
    \| \tilde{R}_{\infty} - \Gamma \|_{L^1(\R)} \underset{t\to \infty}{\longrightarrow} 0.
\end{equation}
Going back to system \eqref{hydro_renormalized}, we need to show that $\tilde{R}_{\infty}$ is independent of $s$. In the $s$ variable, equation \eqref{continuity_renormalized}  becomes
\begin{equation} \label{continuity_phi_s} 
\partial_s \tilde{R} + \partial_y \tilde{M} =0,  
\end{equation}
and \eqref{scaled_equation_lemma_2} implies that $\tilde{M} \in L^2(-1,2;L^1(\R))$. With $\tilde{M}_n(s):= \tilde{M}(s+s_n)$, we have
\[ \partial_y \tilde{M}_n \underset{n\to \infty}{\longrightarrow} 0 \ \ \ \text{in} \ L^2(-1,2;W^{-1,1}(\R)), \]
so
\[ \partial_s \tilde{R}_{\infty}=0. \]
Combining this last equality with equation \eqref{conv_gaussian_infty}, we infer that $\tilde{R}_{\infty}=\Gamma$. The limit being unique, no extraction of a subsequence is needed, and we conclude that 
\[ \tilde{R}(s) \underset{s \to \infty}{\rightharpoonup} \Gamma  \text{ weakly in } L^1(\R). \]

\section{Conclusion}
We have shown that Gaussian functions play an important role in the study of the isothermal compressible Euler equation, standing both as particular solutions of the system \eqref{EL} and time-asymptotic limit of solutions of this system. They also make the link between the isentropic and the isothermal system, as limit of the Barenblatt solutions of \eqref{porous_medium_eq}.\\

The next step of the analysis of the long-time behavior of this system would be to find an explicit convergence rate of any $L^{\infty}$ weak solution of \eqref{EL} towards the limit Gaussian profile, adapting the work from \cite{marcati2005}, \cite{huang2011} and \cite{huang2019}.
A good approach seems to be the use of the Csiszár-Kullback inequality that gives a lower bound of the entropy by the $L^1$-norm of $\rho - \overline{\rho}$:
\[  \| \rho-\overline{\rho} \|^2_{L^1} \leq  2 \| \rho_0 \|_{L^1} \int_{\R} \rho \log \left( \frac{\rho}{\overline{\rho}} \right) = 2 \| \rho_0 \|_{L^1} \int_{\R} R \log \left( \frac{R}{\Gamma}\right). \]

Unfortunately, no decreasing rate of the entropy function is currently known. This is still an open question that appears in other fields of the analysis of PDEs, for instance the study of the logarithmic Schrödinger equation \cite{carles2018}. Note that in \cite{zhao2010}, the author indeed has an upper bound for the entropy, assuming $\overline{\rho} \geq c > 0$, which is true on the compact set $\Omega \subset \R $ but obviously false on the whole space $\R$.

\bibliographystyle{plain}
\bibliography{biblio}

@article {carles2018,
    AUTHOR = {Carles, R\'{e}mi and Gallagher, Isabelle},
     TITLE = {Universal dynamics for the defocusing logarithmic
              {S}chr\"{o}dinger equation},
   JOURNAL = {Duke Math. J.},
  FJOURNAL = {Duke Mathematical Journal},
    VOLUME = {167},
      YEAR = {2018},
    NUMBER = {9},
     PAGES = {1761--1801},
      ISSN = {0012-7094},
   MRCLASS = {35Q55 (35Q40)},
  MRNUMBER = {3813596},
       DOI = {10.1215/00127094-2018-0006},
       URL = {https://doi.org/10.1215/00127094-2018-0006},
}

@article {antonelli2009,
    AUTHOR = {Antonelli, Paolo and Marcati, Pierangelo},
     TITLE = {On the finite energy weak solutions to a system in quantum
              fluid dynamics},
   JOURNAL = {Comm. Math. Phys.},
  FJOURNAL = {Communications in Mathematical Physics},
    VOLUME = {287},
      YEAR = {2009},
    NUMBER = {2},
     PAGES = {657--686},
      ISSN = {0010-3616},
   MRCLASS = {82D50 (35D30 35Q35 76Y05 82D37 82D55)},
  MRNUMBER = {2481754},
       DOI = {10.1007/s00220-008-0632-0},
       URL = {https://doi.org/10.1007/s00220-008-0632-0},
}

@book {cazenave,
    AUTHOR = {Cazenave, Thierry},
     TITLE = {Semilinear {S}chr\"{o}dinger equations},
    SERIES = {Courant Lecture Notes in Mathematics},
    VOLUME = {10},
 PUBLISHER = {New York University, Courant Institute of Mathematical
              Sciences, New York; American Mathematical Society, Providence,
              RI},
      YEAR = {2003},
     PAGES = {xiv+323},
      ISBN = {0-8218-3399-5},
   MRCLASS = {35Q55 (35-01 35J10 35Q40)},
  MRNUMBER = {2002047},
MRREVIEWER = {Woodford W. Zachary},
       DOI = {10.1090/cln/010},
       URL = {https://doi.org/10.1090/cln/010},
}

@article {mensky,
    AUTHOR = {Mensky, Michael},
     TITLE = {Continuous quantum measurements: restricted path integrals and
              master equations},
   JOURNAL = {Phys. Lett. A},
  FJOURNAL = {Physics Letters. A},
    VOLUME = {196},
      YEAR = {1994},
    NUMBER = {3-4},
     PAGES = {159--167},
      ISSN = {0375-9601},
   MRCLASS = {81P15 (81S40)},
  MRNUMBER = {1312330},
MRREVIEWER = {Takashi Miura},
       DOI = {10.1016/0375-9601(94)91219-X},
       URL = {https://doi.org/10.1016/0375-9601(94)91219-X},
}

@article {zander,
    AUTHOR = {Zander, C. and Plastino, A. R. and D\'{\i}az-Alonso, J.},
     TITLE = {Wave packet dynamics for a non-linear {S}chr\"{o}dinger equation
              describing continuous position measurements},
   JOURNAL = {Ann. Physics},
  FJOURNAL = {Annals of Physics},
    VOLUME = {362},
      YEAR = {2015},
     PAGES = {36--56},
      ISSN = {0003-4916},
   MRCLASS = {35Q55 (35C05 78A40 81Q37)},
  MRNUMBER = {3411611},
       DOI = {10.1016/j.aop.2015.07.019},
       URL = {https://doi.org/10.1016/j.aop.2015.07.019},
}

@article{gondran,
  TITLE = {{M{\'e}canique Quantique : Deux interpr{\'e}tations ?}},
  AUTHOR = {Gondran, Alexandre and Gondran, Michel},
  URL = {https://hal.archives-ouvertes.fr/hal-01348957},
  JOURNAL = {{D{\'e}couverte : revue du Palais de la d{\'e}couverte}},
  PUBLISHER = {{Paris : Palais de la d{\'e}couverte}},
  SERIES = {Domestication, une r{\'e}volution {\`a} l'origine de la civilisation},
  NUMBER = {402},
  PAGES = {28-37},
  YEAR = {2016},
  KEYWORDS = {Bohm theory ; Copenhagen School ; Bohmian mechanics ; Interpretation of Quantum Mechanics ; Schrodinger cat ; Interpr{\'e}tation de la m{\'e}canique quantique ; Chat de Schr{\"o}dinger ; M{\'e}canique bohmienne ; Ecole de Copenhague ; Th{\'e}orie de Bohm},
  PDF = {https://hal.archives-ouvertes.fr/hal-01348957/file/revue_palais.pdf},
  HAL_ID = {hal-01348957},
  HAL_VERSION = {v1},
}

@article{kostin,
author = {D. Kostin, M},
year = {1972},
month = {11},
pages = {3589-3591},
title = {On the {S}chrödinger-{L}angevin Equation},
volume = {57},
journal = {The Journal of Chemical Physics},
doi = {10.1063/1.1678812}
}

@book{nassar,
author = {Nassar, Antonio and Miret-Artés, Salvador},
month = {01},
pages = {},
title = {Bohmian Mechanics, Open Quantum Systems and Continuous Measurements},
year = {2017},
publisher={Springer},
isbn = {978-3-319-53651-4},
doi = {10.1007/978-3-319-53653-8}
}

@article{carles2019,
  TITLE = {{Global weak solutions  for quantum isothermal fluids}},
  AUTHOR = {Carles, R{\'e}mi and Carrapatoso, Kleber and Hillairet, Matthieu},
  NOTE = {Preprint, archived at \url{https://arxiv.org/pdf/1905.00732.pdf}},
  YEAR = {2019},
  MONTH = May,
  PDF = {https://hal.archives-ouvertes.fr/hal-02116596/file/existence.pdf},
  HAL_ID = {hal-02116596},
  HAL_VERSION = {v1},
}

@article{carles_mehats2013,
  TITLE = {{An asymptotic preserving scheme based on a new formulation for NLS in the semiclassical limit}},
  AUTHOR = {Besse, Christophe and Carles, R{\'e}mi and M{\'e}hats, Florian},
  URL = {https://hal.archives-ouvertes.fr/hal-00752011},
  JOURNAL = {{Multiscale Modeling and Simulation: A SIAM Interdisciplinary Journal}},
  PUBLISHER = {{Society for Industrial and Applied Mathematics}},
  VOLUME = {11},
  NUMBER = {4},
  PAGES = {1228-1260},
  YEAR = {2013},
  DOI = {10.1137/120899017},
  PDF = {https://hal.archives-ouvertes.fr/hal-00752011/file/ap.pdf},
  HAL_ID = {hal-00752011},
  HAL_VERSION = {v1},
}

@article{jungel2010,
author = {Jüngel, Ansgar},
year = {2010},
month = {01},
pages = {1025-1045},
title = {Global Weak Solutions to Compressible {N}avier–{S}tokes Equations for Quantum Fluids},
volume = {42},
journal = {SIAM J. Math. Analysis},
doi = {10.1137/090776068}
}


@book {wyatt2003,
    AUTHOR = {Wyatt, Robert E.},
     TITLE = {Quantum dynamics with trajectories},
    SERIES = {Interdisciplinary Applied Mathematics},
    VOLUME = {28},
      NOTE = {Introduction to quantum hydrodynamics,
              With contributions by Corey J. Trahan},
 PUBLISHER = {Springer-Verlag, New York},
      YEAR = {2005},
     PAGES = {xxii+405},
      ISBN = {978-0387-22964-5; 0-387-22964-7},
   MRCLASS = {81-08 (76Y05 81-02 81P05 82C80)},
  MRNUMBER = {2138486},
MRREVIEWER = {Sergei A. Nemnyugin},
}

@Article{Bresch2019,
author="Bresch, Didier
and Gisclon, Marguerite
and Lacroix-Violet, Ingrid",
title="On {N}avier--{S}tokes--{K}orteweg and {E}uler--{K}orteweg Systems: Application to Quantum Fluids Models",
journal="Archive for Rational Mechanics and Analysis",
year="2019",
month="Sep",
day="01",
volume="233",
number="3",
pages="975--1025",
issn="1432-0673",
doi="10.1007/s00205-019-01373-w",
url="https://doi.org/10.1007/s00205-019-01373-w"
}

@article{chavanis2017,
  TITLE = {{Derivation of a generalized Schr{\"o}dinger equation from the theory of scale relativity}},
  AUTHOR = {Chavanis, Pierre-Henri},
  URL = {https://hal.archives-ouvertes.fr/hal-01582886},
  JOURNAL = {{Eur.Phys.J.Plus}},
  VOLUME = {132},
  NUMBER = {6},
  PAGES = {286},
  YEAR = {2017},
  DOI = {10.1140/epjp/i2017-11528-3},
  HAL_ID = {hal-01582886},
  HAL_VERSION = {v1},
}

@article{chavanis2019cosmo,
   title={Derivation of the core mass-halo mass relation of fermionic and bosonic dark matter halos from an effective thermodynamical model},
   volume={100},
   ISSN={2470-0029},
   url={http://dx.doi.org/10.1103/PhysRevD.100.123506},
   DOI={10.1103/physrevd.100.123506},
   number={12},
   journal={Physical Review D},
   publisher={American Physical Society (APS)},
   author={Chavanis, Pierre-Henri},
   year={2019},
   month={Dec}
}


@article{chavanis2019stat,
  TITLE = {{Generalized Euler, Smoluchowski and Schr{\"o}dinger equations admitting self-similar solutions with a Tsallis invariant profile}},
  AUTHOR = {Chavanis, Pierre-Henri},
  URL = {https://hal.archives-ouvertes.fr/hal-02283564},
  JOURNAL = {{Eur.Phys.J.Plus}},
  VOLUME = {134},
  NUMBER = {7},
  PAGES = {353},
  YEAR = {2019},
  DOI = {10.1140/epjp/i2019-12706-y},
  HAL_ID = {hal-02283564},
  HAL_VERSION = {v1},
}

@article {ardila2016,
    AUTHOR = {Ardila, Alex H.},
     TITLE = {Orbital stability of {G}ausson solutions to logarithmic
              {S}chr\"{o}dinger equations},
   JOURNAL = {Electron. J. Differential Equations},
  FJOURNAL = {Electronic Journal of Differential Equations},
      YEAR = {2016},
     PAGES = {Paper No. 335, 9},
   MRCLASS = {35Q55 (35A15 35B35)},
  MRNUMBER = {3604780},
MRREVIEWER = {Alessandro Maria Selvitella},
}


@article{birula1976,
title = "Nonlinear wave mechanics",
journal = "Annals of Physics",
volume = "100",
number = "1",
pages = "62 - 93",
year = "1976",
issn = "0003-4916",
doi = "https://doi.org/10.1016/0003-4916(76)90057-9",
url = "http://www.sciencedirect.com/science/article/pii/0003491676900579",
author = "Iwo Bialynicki-Birula and Jerzy Mycielski",
}

@article{arnold2000,
author = {Arnold, Anton and Markowich, Peter and Unterreiter, Andreas},
year = {2000},
month = {05},
pages = {},
title = {On convex {S}obolev inequalities and the rate of convergence to equilibrium for {F}okker-{P}lanck type equations},
volume = {26},
journal = {Communications in Partial Differential Equations},
doi = {10.1081/PDE-100002246}
}

@article {jungel2012,
    AUTHOR = {J\"{u}ngel, Ansgar},
     TITLE = {Dissipative quantum fluid models},
   JOURNAL = {Riv. Math. Univ. Parma (N.S.)},
  FJOURNAL = {Rivista di Matematica della Universit\`a di Parma. New Series. A
              Journal of Pure and Applied Mathematics},
    VOLUME = {3},
      YEAR = {2012},
    NUMBER = {2},
     PAGES = {217--290},
      ISSN = {0035-6298},
   MRCLASS = {82C31 (35Q40 76Y05 82D50)},
  MRNUMBER = {2964097},
MRREVIEWER = {Vittorio Romano},
}

@article {davenia2014,
    AUTHOR = {d'Avenia, Pietro and Montefusco, Eugenio and Squassina, Marco},
     TITLE = {On the logarithmic {S}chr\"{o}dinger equation},
   JOURNAL = {Commun. Contemp. Math.},
  FJOURNAL = {Communications in Contemporary Mathematics},
    VOLUME = {16},
      YEAR = {2014},
    NUMBER = {2},
     PAGES = {1350032, 15},
      ISSN = {0219-1997},
   MRCLASS = {35J91 (35B07 35J20 35Q55)},
  MRNUMBER = {3195154},
MRREVIEWER = {Adilson E. Presoto},
       DOI = {10.1142/S0219199713500326},
       URL = {https://doi.org/10.1142/S0219199713500326},
}

@article{carles_rigidity,
    author = {Carles, R\'emi and Carrapatoso, Kleber and Hillairet, Matthieu},
    title = {Rigidity results in generalized isothermal fluids},
    journal = {Annales Henri Lebesgue},
    publisher = {\'ENS Rennes},
    volume = {1},
    year = {2018},
    pages = {47-85},
    doi = {10.5802/ahl.2},
    language = {en},
    url={ahl.centre-mersenne.org/item/AHL_2018__1__47_0/}
}

@article{li_wang_2006,
title = {Blowup phenomena of solutions to the {E}uler equations for compressible fluid flow},
journal = "Journal of Differential Equations",
volume = "221",
number = "1",
pages = "91 - 101",
year = "2006",
issn = "0022-0396",
doi = "https://doi.org/10.1016/j.jde.2004.12.004",
url = "http://www.sciencedirect.com/science/article/pii/S0022039604005182",
author = "Tianhong Li and Dehua Wang",
keywords = "Euler equations, Compressible fluid, Spherical symmetry, Special solution, Blowup",
abstract = "The blowup phenomena of solutions is investigated for the Euler equations of compressible fluid flow. The approach is to construct special explicit solutions with spherical symmetry to study certain blowup behavior of multi-dimensional solutions. In particular, the special solutions with velocity of the form c(t)x are constructed to show the expanding and blowup properties. The solution with velocity of the form a˙(t)x/a(t) for γ⩾1 and for any space dimensions is obtained as a corollary. Another conclusion is that there is only trivial solution with velocity of the form c(t)|x|α-1x for α≠1 and multi-space dimensions."
}

@article{carles_bao_splitting,
  TITLE = {{Regularized numerical methods for the logarithmic Schrodinger equation}},
  AUTHOR = {Bao, Weizhu and Carles, R{\'e}mi and Su, Chunmei and Tang, Qinglin},
  URL = {https://hal.archives-ouvertes.fr/hal-01937783},
  NOTE = {23 pages, 8 colored figures},
  JOURNAL = {{Numerische Mathematik}},
  PUBLISHER = {{Springer Verlag}},
  VOLUME = {143},
  NUMBER = {2},
  PAGES = {461-487},
  YEAR = {2019},
  DOI = {10.1007/s00211-019-01058-2},
  PDF = {https://hal.archives-ouvertes.fr/hal-01937783/file/numLSE.pdf},
  HAL_ID = {hal-01937783},
  HAL_VERSION = {v1},
}

@article{carles_bao_DF,
  TITLE = {{Error estimates of a regularized finite difference method for the logarithmic Schr{\"o}dinger equation}},
  AUTHOR = {Bao, Weizhu and Carles, R{\'e}mi and Su, Chunmei and Tang, Qinglin},
  URL = {https://hal.archives-ouvertes.fr/hal-01745366},
  JOURNAL = {{SIAM Journal on Numerical Analysis}},
  PUBLISHER = {{Society for Industrial and Applied Mathematics}},
  VOLUME = {57},
  NUMBER = {2},
  PAGES = {657-680},
  YEAR = {2019},
  DOI = {10.1137/18M1177445},
  KEYWORDS = {convergence rate ; error estimates ; semi-implicit finite difference method ; logarithmic nonlinearity ; Logarithmic Schr{\"o}dinger equation ; regularized logarithmic Schr{\"o}dinger equation},
  PDF = {https://hal.archives-ouvertes.fr/hal-01745366/file/LSE-FDM.pdf},
  HAL_ID = {hal-01745366},
  HAL_VERSION = {v1},
}

@article{nassar1985,
	doi = {10.1088/0305-4470/18/9/004},
	url = {https://doi.org/10.1088
	year = 1985,
	month = {jun},
	publisher = {{IOP} Publishing},
	volume = {18},
	number = {9},
	pages = {L509--L511},
	author = {A B Nassar},
	title = {Fluid formulation of a generalised  {S}chr{\"o}dinger-{L}angevin equation},
	journal = {Journal of Physics A: Mathematical and General},
	abstract = {The author recasts a generalised Schrodinger-Langevin equation into a set of hydro-dynamical equations and draws some analogies with a classical (plasma) fluid theory describing the motion of charged particles in a neutralised background.}
}

@article{mousavi2019,
   title={On non-linear {S}chrödinger equations for open quantum systems},
   volume={134},
   pages={1-14},
   ISSN={2190-5444},
   url={http://dx.doi.org/10.1140/epjp/i2019-12965-6},
   DOI={10.1140/epjp/i2019-12965-6},
   number={9},
   journal={The European Physical Journal Plus},
   publisher={Springer Science and Business Media LLC},
   author={Mousavi, S. V. and Miret-Artés, S.},
   year={2019},
   month={Sep}
}

@book{teufel2009,
  title={Bohmian Mechanics: The Physics and Mathematics of Quantum Theory},
  author={D{\"u}rr, D. and Teufel, S.},
  isbn={9783540893448},
  lccn={2009922149},
  url={https://books.google.fr/books?id=UWP2ZSs-UD0C},
  year={2009},
  publisher={Springer Berlin Heidelberg}
}

@article{sparber2014,
title = {W{K}{B} Analysis of {B}ohmian Dynamics},
abstract = "We consider a semiclassically scaled Schr{\"o}dinger equation with WKB initial data. We prove that in the classical limit the corresponding Bohmian trajectories converge (locally in measure) to the classical trajectories before the appearance of the first caustic. In a second step we show that after caustic onset this convergence in general no longer holds. In addition, we provide numerical simulations of the Bohmian trajectories in the semiclassical regime that illustrate the above results.",
author = "A. Figalli and C. Klein and P. Markowich and C. Sparber",
year = "2014",
month = "4",
day = "1",
doi = "10.1002/cpa.21487",
language = "English (US)",
volume = "67",
pages = "581--620",
journal = "Communications on Pure and Applied Mathematics",
issn = "0010-3640",
publisher = "Wiley-Liss Inc.",
number = "4",
}

@article{faou2012,
  TITLE = {{Sobolev stability of plane wave solutions to the cubic nonlinear Schr{\"o}dinger equation on a torus}},
  AUTHOR = {Faou, Erwan and Gauckler, Ludwig and Lubich, Christian},
  URL = {https://hal.archives-ouvertes.fr/hal-00622240},
  JOURNAL = {{Communications in Partial Differential Equations}},
  HAL_LOCAL_REFERENCE = {2011-52},
  PUBLISHER = {{Taylor \& Francis}},
  VOLUME = {38},
  NUMBER = {7},
  PAGES = {1123-1140},
  YEAR = {2013},
  DOI = {10.1080/03605302.2013.785562},
  KEYWORDS = {Nonlinear Schr{\"o}dinger equation ; Normal form ; Modulated Fourier expansion},
  PDF = {https://hal.archives-ouvertes.fr/hal-00622240/file/nlspw-v30.pdf},
  HAL_ID = {hal-00622240},
  HAL_VERSION = {v2},
}


@article {ferriere2019,
    AUTHOR = {Ferriere, Guillaume},
     TITLE = {The focusing logarithmic {S}chr\"{o}dinger equation: analysis of
              breathers and nonlinear superposition},
   JOURNAL = {Discrete Contin. Dyn. Syst.},
  FJOURNAL = {Discrete and Continuous Dynamical Systems. Series A},
    VOLUME = {40},
      YEAR = {2020},
    NUMBER = {11},
     PAGES = {6247--6274},
      ISSN = {1078-0947},
   MRCLASS = {35Q55 (35B10 35B40 35C08)},
  MRNUMBER = {4147349},
       DOI = {10.3934/dcds.2020277},
       URL = {https://doi.org/10.3934/dcds.2020277},
}

@article{cazenave1983,
title = "Stable solutions of the logarithmic Schrödinger equation",
journal = "Nonlinear Analysis: Theory, Methods and Applications",
volume = "7",
number = "10",
pages = "1127 - 1140",
year = "1983",
issn = "0362-546X",
doi = "https://doi.org/10.1016/0362-546X(83)90022-6",
url = "http://www.sciencedirect.com/science/article/pii/0362546X83900226",
author = "Thierry Cazenave",
keywords = "Asymptotic behaviour, orbital stability, Schrödinger equation, stationary states"
}

@article{lasalle1968,
title = "Stability theory for ordinary differential equations",
journal = "Journal of Differential Equations",
volume = "4",
number = "1",
pages = "57 - 65",
year = "1968",
issn = "0022-0396",
doi = "https://doi.org/10.1016/0022-0396(68)90048-X",
url = "http://www.sciencedirect.com/science/article/pii/002203966890048X",
author = "J.P LaSalle"
}


@book{ineg_sobo_log,
    AUTHOR = {An\'{e}, C\'{e}cile and Blach\`ere, S\'{e}bastien and Chafa\"{\i}, Djalil and
              Foug\`eres, Pierre and Gentil, Ivan and Malrieu, Florent and
              Roberto, Cyril and Scheffer, Gr\'{e}gory},
     TITLE = {Sur les in\'{e}galit\'{e}s de {S}obolev logarithmiques},
    SERIES = {Panoramas et Synth\`eses [Panoramas and Syntheses]},
    VOLUME = {10},
      NOTE = {With a preface by Dominique Bakry and Michel Ledoux},
 PUBLISHER = {Soci\'{e}t\'{e} Math\'{e}matique de France, Paris},
      YEAR = {2000},
     PAGES = {xvi+217},
      ISBN = {2-85629-105-8},
   MRCLASS = {46N20 (26D15 46E99 58J60 60J10 60J60)},
  MRNUMBER = {1845806},
MRREVIEWER = {Emmanuel Russ},
}

@book {meyer1966,
    AUTHOR = {Meyer, Paul-A.},
     TITLE = {Probability and potentials},
 PUBLISHER = {Blaisdell Publishing Co. Ginn and Co., Waltham, Mass.-Toronto,
              Ont.-London},
      YEAR = {1966},
     PAGES = {xiii+266},
   MRCLASS = {60.00 (31.00)},
  MRNUMBER = {0205288},
MRREVIEWER = {J. G. Wendel},
}



@article {dunford1938,
    AUTHOR = {Dunford, Nelson},
     TITLE = {Uniformity in linear spaces},
   JOURNAL = {Trans. Amer. Math. Soc.},
  FJOURNAL = {Transactions of the American Mathematical Society},
    VOLUME = {44},
      YEAR = {1938},
    NUMBER = {2},
     PAGES = {305--356},
      ISSN = {0002-9947},
   MRCLASS = {46B99 (46G10 47B10 47B38)},
  MRNUMBER = {1501971},
       DOI = {10.2307/1989974},
       URL = {https://doi.org/10.2307/1989974},
}

@book {lieb_loss,
    AUTHOR = {Lieb, Elliott H. and Loss, Michael},
     TITLE = {Analysis},
    SERIES = {Graduate Studies in Mathematics},
    VOLUME = {14},
   EDITION = {Second},
 PUBLISHER = {American Mathematical Society, Providence, RI},
      YEAR = {2001},
     PAGES = {xxii+346},
      ISBN = {0-8218-2783-9},
   MRCLASS = {00A05 (26-01 28-01 31-01 35J10 42-01)},
  MRNUMBER = {1817225},
       DOI = {10.1090/gsm/014},
       URL = {https://doi.org/10.1090/gsm/014},
}
		
		
@article {jungel2004,
    AUTHOR = {Gualdini, Maria Pia and J\"{u}ngel, Ansgar},
     TITLE = {Analysis of the viscous quantum hydrodynamic equations for
              semiconductors},
   JOURNAL = {European J. Appl. Math.},
  FJOURNAL = {European Journal of Applied Mathematics},
    VOLUME = {15},
      YEAR = {2004},
    NUMBER = {5},
     PAGES = {577--595},
      ISSN = {0956-7925},
   MRCLASS = {82D37 (35Q40)},
  MRNUMBER = {2128612},
MRREVIEWER = {Martin Burger},
       DOI = {10.1017/S0956792504005686},
       URL = {https://doi.org/10.1017/S0956792504005686},
}

@article {mehats2007,
    AUTHOR = {Degond, Pierre and Gallego, Samy and M\'{e}hats, Florian},
     TITLE = {An asymptotic preserving scheme for the {S}chr\"{o}dinger equation
              in the semiclassical limit},
   JOURNAL = {C. R. Math. Acad. Sci. Paris},
  FJOURNAL = {Comptes Rendus Math\'{e}matique. Acad\'{e}mie des Sciences. Paris},
    VOLUME = {345},
      YEAR = {2007},
    NUMBER = {9},
     PAGES = {531--536},
      ISSN = {1631-073X},
   MRCLASS = {65M06 (76A02)},
  MRNUMBER = {2375117},
MRREVIEWER = {Penny J. Davies},
       DOI = {10.1016/j.crma.2007.10.014},
       URL = {https://doi.org/10.1016/j.crma.2007.10.014},
}


@incollection {grebert2007,
    AUTHOR = {Gr\'{e}bert, Beno\^{\i}t},
     TITLE = {Birkhoff normal form and {H}amiltonian {PDE}s},
 BOOKTITLE = {Partial differential equations and applications},
    SERIES = {S\'{e}min. Congr.},
    VOLUME = {15},
     PAGES = {1--46},
 PUBLISHER = {Soc. Math. France, Paris},
      YEAR = {2007},
   MRCLASS = {37K55 (35F20 37J40 37K05)},
  MRNUMBER = {2352816},
MRREVIEWER = {Yannick Sire},
}

@article {bambusi2006,
    AUTHOR = {Bambusi, D. and Gr\'{e}bert, B.},
     TITLE = {Birkhoff normal form for partial differential equations with
              tame modulus},
   JOURNAL = {Duke Math. J.},
  FJOURNAL = {Duke Mathematical Journal},
    VOLUME = {135},
      YEAR = {2006},
    NUMBER = {3},
     PAGES = {507--567},
      ISSN = {0012-7094},
   MRCLASS = {37K55 (35Q55)},
  MRNUMBER = {2272975},
MRREVIEWER = {C. Eugene Wayne},
       DOI = {10.1215/S0012-7094-06-13534-2},
       URL = {https://doi.org/10.1215/S0012-7094-06-13534-2},
}

@article {vasseur2016,
    AUTHOR = {Vasseur, Alexis F. and Yu, Cheng},
     TITLE = {Global weak solutions to the compressible quantum
              {N}avier-{S}tokes equations with damping},
   JOURNAL = {SIAM J. Math. Anal.},
  FJOURNAL = {SIAM Journal on Mathematical Analysis},
    VOLUME = {48},
      YEAR = {2016},
    NUMBER = {2},
     PAGES = {1489--1511},
      ISSN = {0036-1410},
   MRCLASS = {35Q35 (35D30 76N10 76Y05)},
  MRNUMBER = {3490496},
MRREVIEWER = {Reinhard Farwig},
       DOI = {10.1137/15M1013730},
       URL = {https://doi.org/10.1137/15M1013730},
}

@misc{bresch2020,
    title={On the Exponential decay for Compressible {N}avier-{S}tokes-{K}orteweg equations with a Drag Term},
    author={Didier Bresch and Marguerite Gisclon and Ingrid Lacroix-Violet and Alexis Vasseur},
    howpublished={Preprint, archived at \url{https://arxiv.org/pdf/2004.07895.pdf}},
    year={2020},
    eprint={2004.07895},
    archivePrefix={arXiv},
    primaryClass={math.AP}
}

@incollection {feireisl2014,
    AUTHOR = {Feireisl, Eduard},
     TITLE = {Relative entropies, dissipative solutions, and singular limits
              of complete fluid systems},
 BOOKTITLE = {Hyperbolic problems: theory, numerics, applications},
    SERIES = {AIMS Ser. Appl. Math.},
    VOLUME = {8},
     PAGES = {11--27},
 PUBLISHER = {Am. Inst. Math. Sci. (AIMS), Springfield, MO},
      YEAR = {2014},
   MRCLASS = {35Q35 (35B25 35Q30 35Q31)},
  MRNUMBER = {3524312},
}

@article {vasseur2018,
    AUTHOR = {Lacroix-Violet, Ingrid and Vasseur, Alexis},
     TITLE = {Global weak solutions to the compressible quantum
              {N}avier-{S}tokes equation and its semi-classical limit},
   JOURNAL = {J. Math. Pures Appl. (9)},
  FJOURNAL = {Journal de Math\'{e}matiques Pures et Appliqu\'{e}es. Neuvi\`eme S\'{e}rie},
    VOLUME = {114},
      YEAR = {2018},
     PAGES = {191--210},
      ISSN = {0021-7824},
   MRCLASS = {35Q35 (35A01 35D30 35Q30 35Q40 76N10 76Y05)},
  MRNUMBER = {3801754},
MRREVIEWER = {Nikolay G. Kuznetsov},
       DOI = {10.1016/j.matpur.2017.12.002},
       URL = {https://doi.org/10.1016/j.matpur.2017.12.002},
}

@book {lions1996,
    AUTHOR = {Lions, Pierre-Louis},
     TITLE = {Mathematical topics in fluid mechanics. {V}ol. 1},
    SERIES = {Oxford Lecture Series in Mathematics and its Applications},
    VOLUME = {3},
      NOTE = {Incompressible models,
              Oxford Science Publications},
 PUBLISHER = {The Clarendon Press, Oxford University Press, New York},
      YEAR = {1996},
     PAGES = {xiv+237},
      ISBN = {0-19-851487-5},
   MRCLASS = {76-02 (35Q30 35Q35 76D05)},
  MRNUMBER = {1422251},
MRREVIEWER = {Denis Serre},
}

@article{bresch2015,
  TITLE = {{A generalization of the quantum Bohm identity: Hyperbolic CFL condition for Euler--Korteweg equations}},
  AUTHOR = {Bresch, Didier and Couderc, Fr{\'e}d{\'e}ric and Noble, Pascal and Vila, Jean-Paul},
  URL = {https://hal.archives-ouvertes.fr/hal-01870745},
  JOURNAL = {{Comptes Rendus Math{\'e}matique}},
  PUBLISHER = {{Elsevier Masson}},
  YEAR = {2015},
  MONTH = Nov,
  DOI = {10.1016/j.crma.2015.09.020},
  PDF = {https://hal.archives-ouvertes.fr/hal-01870745/file/1-s2.0-S1631073X15002460-main.pdf},
  HAL_ID = {hal-01870745},
  HAL_VERSION = {v1},
}
	
@article {bresch2003,
    AUTHOR = {Bresch, Didier and Desjardins, Beno\^{\i}t and Lin, Chi-Kun},
     TITLE = {On some compressible fluid models: {K}orteweg, lubrication,
              and shallow water systems},
   JOURNAL = {Comm. Partial Differential Equations},
  FJOURNAL = {Communications in Partial Differential Equations},
    VOLUME = {28},
      YEAR = {2003},
    NUMBER = {3-4},
     PAGES = {843--868},
      ISSN = {0360-5302},
   MRCLASS = {76N10 (35Q35 76D45)},
  MRNUMBER = {1978317},
MRREVIEWER = {Denis Serre},
       DOI = {10.1081/PDE-120020499},
       URL = {https://doi.org/10.1081/PDE-120020499},
}

@article{bresch2004,
title = "Quelques modèles diffusifs capillaires de type {K}orteweg",
journal = "Comptes Rendus Mécanique",
volume = "332",
number = "11",
pages = "881 - 886",
year = "2004",
issn = "1631-0721",
doi = "https://doi.org/10.1016/j.crme.2004.07.003",
url = "http://www.sciencedirect.com/science/article/pii/S1631072104001809",
author = "Didier Bresch and Benoît Desjardins",
keywords = "Mécanique des fluides numérique, Modèles d'interface diffuse, Modèles de type Korteweg, Estimations d'énergie, Modèles compressibles, Computational fluid mechanics, Diffuse interface models, Korteweg models, Energy estimates, Compressible flows",
}


@misc{ferriere2020,
    title={{W}{K}{B} analysis and semiclassical limit of the Logarithmic Non-linear {S}chr\"{o}dinger Equation in an analytic framework},
    author={Guillaume Ferriere},
    year={2020},
    note={Work in progress}
}

@article{bresch2006,
    AUTHOR = {Bresch, Didier and Desjardins, Beno\^{\i}t},
     TITLE = {On the construction of approximate solutions for the 2{D}
              viscous shallow water model and for compressible
              {N}avier-{S}tokes models},
   JOURNAL = {J. Math. Pures Appl. (9)},
  FJOURNAL = {Journal de Math\'{e}matiques Pures et Appliqu\'{e}es. Neuvi\`eme S\'{e}rie},
    VOLUME = {86},
      YEAR = {2006},
    NUMBER = {4},
     PAGES = {362--368},
      ISSN = {0021-7824},
   MRCLASS = {35Q35 (35D05 76B03 76N10)},
  MRNUMBER = {2257849},
MRREVIEWER = {Drago\c{s} Iftimie},
       DOI = {10.1016/j.matpur.2006.06.005},
       URL = {https://doi.org/10.1016/j.matpur.2006.06.005},
}

@article {bresch2015_2,
    AUTHOR = {Bresch, Didier and Desjardins, Beno\^{\i}t and Zatorska, Ewelina},
     TITLE = {Two-velocity hydrodynamics in fluid mechanics: {P}art {II}.
              {E}xistence of global {$\kappa$}-entropy solutions to the
              compressible {N}avier-{S}tokes systems with degenerate
              viscosities},
   JOURNAL = {J. Math. Pures Appl. (9)},
  FJOURNAL = {Journal de Math\'{e}matiques Pures et Appliqu\'{e}es. Neuvi\`eme S\'{e}rie},
    VOLUME = {104},
      YEAR = {2015},
    NUMBER = {4},
     PAGES = {801--836},
      ISSN = {0021-7824},
   MRCLASS = {35Q35 (35B10 35B30 35B40 35D30 35Q30 76N10)},
  MRNUMBER = {3394617},
MRREVIEWER = {C\'{e}sar J. Niche},
       DOI = {10.1016/j.matpur.2015.05.004},
       URL = {https://doi.org/10.1016/j.matpur.2015.05.004},
}

@incollection {bresch2018,
    AUTHOR = {Bresch, Didier and Noble, Pascal and Vila, Jean-Paul},
     TITLE = {Relative entropy for compressible {N}avier-{S}tokes equations
              with density dependent viscosities and various applications},
 BOOKTITLE = {L{MLFN} 2015---low velocity flows---application to low {M}ach
              and low {F}roude regimes},
    SERIES = {ESAIM Proc. Surveys},
    VOLUME = {58},
     PAGES = {40--57},
 PUBLISHER = {EDP Sci., Les Ulis},
      YEAR = {2017},
   MRCLASS = {35Q35 (35A02 76N10)},
  MRNUMBER = {3734208},
       DOI = {10.1051/proc/201758040},
       URL = {https://doi.org/10.1051/proc/201758040},
}

@misc{cigler2014,
      title={Continuous q-Hermite polynomials: An elementary approach}, 
      author={Johann Cigler},
      year={2014},
      eprint={1307.0357},
      archivePrefix={arXiv},
      primaryClass={math.HO}
}

@article{odake2011,
   title={Discrete quantum mechanics},
   volume={44},
   ISSN={1751-8121},
   url={http://dx.doi.org/10.1088/1751-8113/44/35/353001},
   DOI={10.1088/1751-8113/44/35/353001},
   number={35},
   journal={Journal of Physics A: Mathematical and Theoretical},
   publisher={IOP Publishing},
   author={Odake, Satoru and Sasaki, Ryu},
   year={2011},
   month={Aug},
   pages={353001}
}

@article{chauleur_temps_long,
  TITLE = {{Dynamics of the Schr{\"o}dinger-Langevin equation}},
  AUTHOR = {Chauleur, Quentin},
  URL = {https://hal.archives-ouvertes.fr/hal-02541831},
  JOURNAL = {{Nonlinearity}},
  PUBLISHER = {{IOP Publishing}},
  VOLUME = {34},
  NUMBER = {4},
  PAGES = {1943-1974},
  YEAR = {2021},
  DOI = {10.1088/1361-6544/abd528},
  PDF = {https://hal.archives-ouvertes.fr/hal-02541831/file/dynamique_gaussienne.pdf},
  HAL_ID = {hal-02541831},
  HAL_VERSION = {v1},
}

@article {marcati2005,
    AUTHOR = {Huang, Feimin and Marcati, Pierangelo and Pan, Ronghua},
     TITLE = {Convergence to the {B}arenblatt solution for the compressible
              {E}uler equations with damping and vacuum},
   JOURNAL = {Arch. Ration. Mech. Anal.},
  FJOURNAL = {Archive for Rational Mechanics and Analysis},
    VOLUME = {176},
      YEAR = {2005},
    NUMBER = {1},
     PAGES = {1--24},
      ISSN = {0003-9527},
   MRCLASS = {76N10 (35L65 35Q35 76S05)},
  MRNUMBER = {2185857},
MRREVIEWER = {Yue-Jun Peng},
       DOI = {10.1007/s00205-004-0349-y},
       URL = {https://doi.org/10.1007/s00205-004-0349-y},
}

@incollection {aronson1986,
    AUTHOR = {Aronson, D. G.},
     TITLE = {The porous medium equation},
 BOOKTITLE = {Nonlinear diffusion problems ({M}ontecatini {T}erme, 1985)},
    SERIES = {Lecture Notes in Math.},
    VOLUME = {1224},
     PAGES = {1--46},
 PUBLISHER = {Springer, Berlin},
      YEAR = {1986},
   MRCLASS = {35K65 (35K55 76N15 76S05)},
  MRNUMBER = {877986},
       DOI = {10.1007/BFb0072687},
       URL = {https://doi.org/10.1007/BFb0072687},
}

@article {lasser2020,
    AUTHOR = {Lasser, Caroline and Lubich, Christian},
     TITLE = {Computing quantum dynamics in the semiclassical regime},
   JOURNAL = {Acta Numer.},
  FJOURNAL = {Acta Numerica},
    VOLUME = {29},
      YEAR = {2020},
     PAGES = {229--401},
      ISSN = {0962-4929},
   MRCLASS = {81-08 (65Mxx 81Q20 81V35)},
  MRNUMBER = {4189293},
       DOI = {10.1017/s0962492920000033},
       URL = {https://doi.org/10.1017/s0962492920000033},
}


@article {zhao2010,
    AUTHOR = {Zhao, Kun},
     TITLE = {On the isothermal compressible {E}uler equations with
              frictional damping},
   JOURNAL = {Commun. Math. Anal.},
  FJOURNAL = {Communications in Mathematical Analysis},
    VOLUME = {9},
      YEAR = {2010},
    NUMBER = {2},
     PAGES = {77--97},
   MRCLASS = {35Q35 (35A09 35B40 35D30 76N10)},
  MRNUMBER = {2737756},
MRREVIEWER = {Francesca Brini},
}

@article {huang2006,
    AUTHOR = {Huang, Feimin and Pan, Ronghua},
     TITLE = {Asymptotic behavior of the solutions to the damped
              compressible {E}uler equations with vacuum},
   JOURNAL = {J. Differential Equations},
  FJOURNAL = {Journal of Differential Equations},
    VOLUME = {220},
      YEAR = {2006},
    NUMBER = {1},
     PAGES = {207--233},
      ISSN = {0022-0396},
   MRCLASS = {35Q35 (35B40 35D05 76N10)},
  MRNUMBER = {2182086},
MRREVIEWER = {Xiaoming Wang},
       DOI = {10.1016/j.jde.2005.03.012},
       URL = {https://doi.org/10.1016/j.jde.2005.03.012},
}

@book {nishida1978,
    AUTHOR = {Nishida, Takaaki},
     TITLE = {Nonlinear hyperbolic equations and related topics in fluid
              dynamics},
    SERIES = {Publications Math\'{e}matiques d'Orsay, No. 78-02},
 PUBLISHER = {D\'{e}partement de Math\'{e}matique, Universit\'{e} de Paris-Sud, Orsay},
      YEAR = {1978},
     PAGES = {iv+123},
   MRCLASS = {35L60 (76.35)},
  MRNUMBER = {0481578},
MRREVIEWER = {J. Smoller},
}

@article {barenblatt1953,
    AUTHOR = {Barenblatt, G. I.},
     TITLE = {On a class of exact solutions of the plane one-dimensional
              problem of unsteady filtration of a gas in a porous medium},
   JOURNAL = {Akad. Nauk SSSR. Prikl. Mat. Meh.},
    VOLUME = {17},
      YEAR = {1953},
     PAGES = {739--742},
   MRCLASS = {76.1X},
  MRNUMBER = {0064556},
MRREVIEWER = {R. E. Gaskell},
}

@article {huang2011,
    AUTHOR = {Huang, Feimin and Pan, Ronghua and Wang, Zhen},
     TITLE = {{$L^1$} convergence to the {B}arenblatt solution for
              compressible {E}uler equations with damping},
   JOURNAL = {Arch. Ration. Mech. Anal.},
  FJOURNAL = {Archive for Rational Mechanics and Analysis},
    VOLUME = {200},
      YEAR = {2011},
    NUMBER = {2},
     PAGES = {665--689},
      ISSN = {0003-9527},
   MRCLASS = {76N10 (35Q30 76S05)},
  MRNUMBER = {2787593},
MRREVIEWER = {Tom\'{a}\v{s} F\"{u}rst},
       DOI = {10.1007/s00205-010-0355-1},
       URL = {https://doi.org/10.1007/s00205-010-0355-1},
}

@article {huang2019,
    AUTHOR = {Geng, Shifeng and Huang, Feimin},
     TITLE = {{$L^1$}-convergence rates to the {B}arenblatt solution for the
              damped compressible {E}uler equations},
   JOURNAL = {J. Differential Equations},
  FJOURNAL = {Journal of Differential Equations},
    VOLUME = {266},
      YEAR = {2019},
    NUMBER = {12},
     PAGES = {7890--7908},
      ISSN = {0022-0396},
   MRCLASS = {35L65 (35K65 35Q35 76N15)},
  MRNUMBER = {3944244},
       DOI = {10.1016/j.jde.2018.12.016},
       URL = {https://doi.org/10.1016/j.jde.2018.12.016},
}

@article {liu1996,
    AUTHOR = {Liu, Tai-Ping},
     TITLE = {Compressible flow with damping and vacuum},
   JOURNAL = {Japan J. Indust. Appl. Math.},
  FJOURNAL = {Japan Journal of Industrial and Applied Mathematics},
    VOLUME = {13},
      YEAR = {1996},
    NUMBER = {1},
     PAGES = {25--32},
      ISSN = {0916-7005},
   MRCLASS = {35Q35 (76N10 76S05)},
  MRNUMBER = {1377457},
MRREVIEWER = {Alberto Valli},
       DOI = {10.1007/BF03167296},
       URL = {https://doi.org/10.1007/BF03167296},
}

@article{chauleur_existence,
author = {Chauleur, Quentin},
year = {2021},
month = {03},
pages = {1-29},
title = {Global dissipative solutions of the defocusing isothermal {E}uler–{L}angevin–{K}orteweg equations},
journal = {Asymptotic Analysis},
doi = {10.3233/ASY-211681}
}

@article{chauleur_barenblatt,
author = {Chauleur, Quentin},
year = {},
month = {},
pages = {},
title = {The isothermal limit for the compressible {E}uler equations with damping},
year={2021},
note={Work in progress}
}
	
@book {abramowitz64,
    AUTHOR = {Abramowitz, Milton and Stegun, Irene A.},
     TITLE = {Handbook of mathematical functions with formulas, graphs, and
              mathematical tables},
    SERIES = {National Bureau of Standards Applied Mathematics Series, No.
              55},
      NOTE = {For sale by the Superintendent of Documents},
 PUBLISHER = {U. S. Government Printing Office, Washington, D. C.},
      YEAR = {1964},
     PAGES = {xiv+1046},
   MRCLASS = {33.00 (65.05)},
  MRNUMBER = {0167642},
MRREVIEWER = {D. H. Lehmer},
}

@article {lefloch2005,
    AUTHOR = {LeFloch, Philippe G. and Shelukhin, Vladimir},
     TITLE = {Symmetries and global solvability of the isothermal gas
              dynamics equations},
   JOURNAL = {Arch. Ration. Mech. Anal.},
  FJOURNAL = {Archive for Rational Mechanics and Analysis},
    VOLUME = {175},
      YEAR = {2005},
    NUMBER = {3},
     PAGES = {389--430},
      ISSN = {0003-9527},
   MRCLASS = {35L65 (35Q35 76M60 76N10 76N15)},
  MRNUMBER = {2126635},
MRREVIEWER = {Alberto Valli},
       DOI = {10.1007/s00205-004-0344-3},
       URL = {https://doi.org/10.1007/s00205-004-0344-3},
}


@book {chung1979,
    AUTHOR = {Chung, Kai Lai},
     TITLE = {Elementary probability theory with stochastic processes},
    SERIES = {Undergraduate Texts in Mathematics},
   EDITION = {Third},
 PUBLISHER = {Springer-Verlag, New York-Heidelberg},
      YEAR = {1979},
     PAGES = {xvi+325},
      ISBN = {3-540-90362-3},
   MRCLASS = {60-01},
  MRNUMBER = {560506},
}

@incollection {marcati2000,
    AUTHOR = {Marcati, Pierangelo and Pan, Ronghua},
     TITLE = {Cauchy problem for compressible {E}uler equations with
              damping},
 BOOKTITLE = {International {C}onference on {D}ifferential {E}quations,
              {V}ol. 1, 2 ({B}erlin, 1999)},
     PAGES = {315--317},
 PUBLISHER = {World Sci. Publ., River Edge, NJ},
      YEAR = {2000},
   MRCLASS = {35Q35 (76N10)},
  MRNUMBER = {1870144},
}

@incollection {huang2000,
    AUTHOR = {Huang, Feimin},
     TITLE = {Large time behavior for compressible {E}uler equations with
              damping and vacuum},
      NOTE = {Mathematical analysis in fluid and gas dynamics (Japanese)
              (Kyoto, 2001)},
   JOURNAL = {S\={u}rikaisekikenky\={u}sho K\B{o}ky\={u}roku},
  FJOURNAL = {S\={u}rikaisekikenky\={u}sho K\B{o}ky\={u}roku},
    NUMBER = {1247},
      YEAR = {2002},
     PAGES = {57--66},
   MRCLASS = {76N10 (35B40 35Q35)},
  MRNUMBER = {1919351},
}

@book {dunford1958,
    AUTHOR = {Dunford, Nelson and Schwartz, Jacob T.},
     TITLE = {Linear {O}perators. {I}. {G}eneral {T}heory},
    SERIES = {Pure and Applied Mathematics, Vol. 7},
 PUBLISHER = {Interscience Publishers, Inc., New York; Interscience
              Publishers, Ltd., London},
      YEAR = {1958},
     PAGES = {xiv+858},
   MRCLASS = {46.00},
  MRNUMBER = {0117523},
MRREVIEWER = {E. H. Rothe},
}



@book {dellacherie1975,
    AUTHOR = {Dellacherie, Claude and Meyer, Paul-Andr\'{e}},
     TITLE = {Probabilit\'{e}s et potentiel},
    SERIES = {Publications de l'Institut de Math\'{e}matique de l'Universit\'{e} de
              Strasbourg, No. XV},
 PUBLISHER = {Hermann, Paris},
      YEAR = {1975},
     PAGES = {x+291},
   MRCLASS = {60-02 (28-02 28A05 31C15)},
  MRNUMBER = {0488194},
MRREVIEWER = {M. G. Sur},
}
		
	

\begin{thebibliography}{10}

\bibitem{abramowitz64}
Milton Abramowitz and Irene~A. Stegun.
\newblock {\em Handbook of mathematical functions with formulas, graphs, and
  mathematical tables}.
\newblock National Bureau of Standards Applied Mathematics Series, No. 55. U.
  S. Government Printing Office, Washington, D. C., 1964.
\newblock For sale by the Superintendent of Documents.

\bibitem{ineg_sobo_log}
C\'{e}cile An\'{e}, S\'{e}bastien Blach\`ere, Djalil Chafa\"{\i}, Pierre
  Foug\`eres, Ivan Gentil, Florent Malrieu, Cyril Roberto, and Gr\'{e}gory
  Scheffer.
\newblock {\em Sur les in\'{e}galit\'{e}s de {S}obolev logarithmiques},
  volume~10 of {\em Panoramas et Synth\`eses [Panoramas and Syntheses]}.
\newblock Soci\'{e}t\'{e} Math\'{e}matique de France, Paris, 2000.
\newblock With a preface by Dominique Bakry and Michel Ledoux.

\bibitem{arnold2000}
Anton Arnold, Peter Markowich, and Andreas Unterreiter.
\newblock On convex {S}obolev inequalities and the rate of convergence to
  equilibrium for {F}okker-{P}lanck type equations.
\newblock {\em Communications in Partial Differential Equations}, 26, 05 2000.

\bibitem{aronson1986}
D.~G. Aronson.
\newblock The porous medium equation.
\newblock In {\em Nonlinear diffusion problems ({M}ontecatini {T}erme, 1985)},
  volume 1224 of {\em Lecture Notes in Math.}, pages 1--46. Springer, Berlin,
  1986.

\bibitem{barenblatt1953}
G.~I. Barenblatt.
\newblock On a class of exact solutions of the plane one-dimensional problem of
  unsteady filtration of a gas in a porous medium.
\newblock {\em Akad. Nauk SSSR. Prikl. Mat. Meh.}, 17:739--742, 1953.

\bibitem{carles_rigidity}
R\'emi Carles, Kleber Carrapatoso, and Matthieu Hillairet.
\newblock Rigidity results in generalized isothermal fluids.
\newblock {\em Annales Henri Lebesgue}, 1:47--85, 2018.

\bibitem{carles2018}
R\'{e}mi Carles and Isabelle Gallagher.
\newblock Universal dynamics for the defocusing logarithmic {S}chr\"{o}dinger
  equation.
\newblock {\em Duke Math. J.}, 167(9):1761--1801, 2018.

\bibitem{chauleur_temps_long}
Quentin Chauleur.
\newblock {Dynamics of the Schr{\"o}dinger-Langevin equation}.
\newblock {\em {Nonlinearity}}, 34(4):1943--1974, 2021.

\bibitem{dellacherie1975}
Claude Dellacherie and Paul-Andr\'{e} Meyer.
\newblock {\em Probabilit\'{e}s et potentiel}.
\newblock Publications de l'Institut de Math\'{e}matique de l'Universit\'{e} de
  Strasbourg, No. XV. Hermann, Paris, 1975.

\bibitem{dunford1958}
Nelson Dunford and Jacob~T. Schwartz.
\newblock {\em Linear {O}perators. {I}. {G}eneral {T}heory}.
\newblock Pure and Applied Mathematics, Vol. 7. Interscience Publishers, Inc.,
  New York; Interscience Publishers, Ltd., London, 1958.

\bibitem{huang2019}
Shifeng Geng and Feimin Huang.
\newblock {$L^1$}-convergence rates to the {B}arenblatt solution for the damped
  compressible {E}uler equations.
\newblock {\em J. Differential Equations}, 266(12):7890--7908, 2019.

\bibitem{huang2000}
Feimin Huang.
\newblock Large time behavior for compressible {E}uler equations with damping
  and vacuum.
\newblock Number 1247, pages 57--66. 2002.
\newblock Mathematical analysis in fluid and gas dynamics (Japanese) (Kyoto,
  2001).

\bibitem{marcati2005}
Feimin Huang, Pierangelo Marcati, and Ronghua Pan.
\newblock Convergence to the {B}arenblatt solution for the compressible {E}uler
  equations with damping and vacuum.
\newblock {\em Arch. Ration. Mech. Anal.}, 176(1):1--24, 2005.

\bibitem{huang2006}
Feimin Huang and Ronghua Pan.
\newblock Asymptotic behavior of the solutions to the damped compressible
  {E}uler equations with vacuum.
\newblock {\em J. Differential Equations}, 220(1):207--233, 2006.

\bibitem{huang2011}
Feimin Huang, Ronghua Pan, and Zhen Wang.
\newblock {$L^1$} convergence to the {B}arenblatt solution for compressible
  {E}uler equations with damping.
\newblock {\em Arch. Ration. Mech. Anal.}, 200(2):665--689, 2011.

\bibitem{lefloch2005}
Philippe~G. LeFloch and Vladimir Shelukhin.
\newblock Symmetries and global solvability of the isothermal gas dynamics
  equations.
\newblock {\em Arch. Ration. Mech. Anal.}, 175(3):389--430, 2005.

\bibitem{li_wang_2006}
Tianhong Li and Dehua Wang.
\newblock Blowup phenomena of solutions to the {E}uler equations for
  compressible fluid flow.
\newblock {\em Journal of Differential Equations}, 221(1):91 -- 101, 2006.

\bibitem{liu1996}
Tai-Ping Liu.
\newblock Compressible flow with damping and vacuum.
\newblock {\em Japan J. Indust. Appl. Math.}, 13(1):25--32, 1996.

\bibitem{marcati2000}
Pierangelo Marcati and Ronghua Pan.
\newblock Cauchy problem for compressible {E}uler equations with damping.
\newblock In {\em International {C}onference on {D}ifferential {E}quations,
  {V}ol. 1, 2 ({B}erlin, 1999)}, pages 315--317. World Sci. Publ., River Edge,
  NJ, 2000.

\bibitem{nishida1978}
Takaaki Nishida.
\newblock {\em Nonlinear hyperbolic equations and related topics in fluid
  dynamics}.
\newblock Publications Math\'{e}matiques d'Orsay, No. 78-02. D\'{e}partement de
  Math\'{e}matique, Universit\'{e} de Paris-Sud, Orsay, 1978.

\bibitem{zhao2010}
Kun Zhao.
\newblock On the isothermal compressible {E}uler equations with frictional
  damping.
\newblock {\em Commun. Math. Anal.}, 9(2):77--97, 2010.

\end{thebibliography}

\end{document}